\documentclass[onefignum,onetabnum]{siamart171218}

\usepackage{lipsum}
\usepackage{amsfonts}
\usepackage{graphicx}

\usepackage{epstopdf}
\usepackage{algorithm}
\usepackage{algorithmic}

\usepackage{multirow}
\usepackage{listings}
\usepackage{mathtools}
\usepackage{latexsym,amsmath,amsfonts,amscd}
\usepackage{subfigure}

\usepackage{cases}
\usepackage{bbm}

\newtheorem{remark}[theorem]{Remark} 

\ifpdf
  \DeclareGraphicsExtensions{.eps,.pdf,.png,.jpg}
\else
  \DeclareGraphicsExtensions{.eps}
\fi

\headers{superconvergence of FD schemes based on variation form}{H. Li and X. Zhang}

\title{Superconvergence of high order finite difference schemes based on variational formulation for elliptic equations   
\thanks{H. Li and X. Zhang  were supported by the NSF grant DMS-1522593 and DMS-1913120.}
}

\author{Hao Li\thanks{Department of Mathematics,
Purdue University,
150 N. University Street,
West Lafayette, IN 47907-2067
  (\email{li2497@purdue.edu}, \email{zhan1966@purdue.edu}).}
\and Xiangxiong Zhang \footnotemark[2]}

\usepackage{amsopn}

\begin{document}
\maketitle

\begin{abstract}
The classical continuous finite element method with Lagrangian $Q^k$ basis reduces to a finite difference scheme when all the integrals are replaced by 
the $(k+1)\times (k+1)$ Gauss-Lobatto quadrature.
We prove that this finite difference scheme is $(k+2)$-th order accurate in the discrete 2-norm for an elliptic equation with   Dirichlet boundary conditions, which is a superconvergence result of function values. 
\end{abstract}
\begin{keywords}
Superconvergence, high order accurate discrete Laplacian, elliptic equations, finite difference formulation based on variational formulation, Gauss-Lobatto quadrature.
\end{keywords}
\begin{AMS}
 	65N30,   	65N15,	65N06
\end{AMS}

\section{Introduction}
\subsection{Motivation}
In this paper we consider solving a two-dimensional elliptic equation with smooth coefficients on 
a rectangular domain by high order finite difference schemes, which are constructed via using  suitable quadrature in the classical continuous finite element method on a rectangular mesh. Consider the following model problem as an example: a variable coefficient Poisson equation $-\nabla(a(\mathbf x)\nabla u)=f, a(\mathbf x)>0$ on a square 
domain $\Omega=(0,1)\times(0,1)$ with homogeneous Dirichlet boundary conditions. 
The  variational form is to find $u\in H_0^1(\Omega)=\{v\in H^1(\Omega): v|_{\partial \Omega}=0\}$ satisfying
\[
 A(u,v)=(f,v),\quad \forall v\in H_0^1(\Omega),
\]
 where  
$A(u,v)=\iint_{\Omega} a\nabla u \cdot \nabla v dx dy$, $ (f,v)=\iint_{\Omega}fv dxdy.$
Let $h$ be the mesh size of an uniform rectangular mesh and  $V_0^h\subseteq H^1_0(\Omega)$ be the continuous finite element space consisting of piecewise $Q^k$ polynomials (i.e., tensor product of piecewise polynomials of degree $k$), then the $C^0$-$Q^k$ finite element solution is defined as $u_h\in V_0^h$ satisfying 
\begin{equation}\label{intro-scheme1}
 A(u_h,v_h)=(f,v_h),\quad \forall v_h\in V_0^h.\end{equation}

  Standard error estimates of \eqref{intro-scheme1} are
 $\|u-u_h\|_{1}\leq C h^{k}\|u\|_{k+1}$ and $\|u-u_h\|_{0}\leq C h^{k+1}\|u\|_{k+1}$ where $\|\cdot\|_k$ denotes $H^k(\Omega)$-norm, 
 see \cite{ciarlet1991basic}.
 For $k\geq 2$, $\mathcal O (h^{k+1})$ superconvergence for the gradient at Gauss quadrature points and $\mathcal O (h^{k+2})$ superconvergence for functions values at Gauss-Lobatto quadrature points were proven for one-dimensional case in  \cite{lesaint1979superconvergence, chen1979superconvergent, bakker1982note}
 and for two-dimensional case in \cite{douglas1974estimate, wahlbin2006superconvergence, chen2001structure, lin1996}.
 
When implementing the scheme \eqref{intro-scheme1}, integrals are usually approximated by quadrature.  The most convenient implementation is to use $(k+1)\times (k+1)$ Gauss-Lobatto quadrature because they not only are superconvergence points but also can define all the degree of freedoms of Lagrangian $Q^k$ basis.  See Figure \ref{mesh} for the case $k=2$. 
Such a quadrature scheme can be denoted as finding $u_h\in V_0^h$ satisfying
\begin{equation}\label{intro-scheme2}
 A_h(u_h,v_h)=\langle f,v_h\rangle_h,\quad \forall v_h\in V_0^h,\end{equation}
where $A_h(u_h,v_h)$  and $\langle f,v_h\rangle_h$ denote using tensor product of $(k+1)$-point Gauss-Lobatto quadrature for integrals $A(u_h,v_h)$ and $(f,v_h)$ respectively.

  \begin{figure}[h]
 \subfigure[The quadrature points and a FEM mesh]{\includegraphics[scale=0.8]{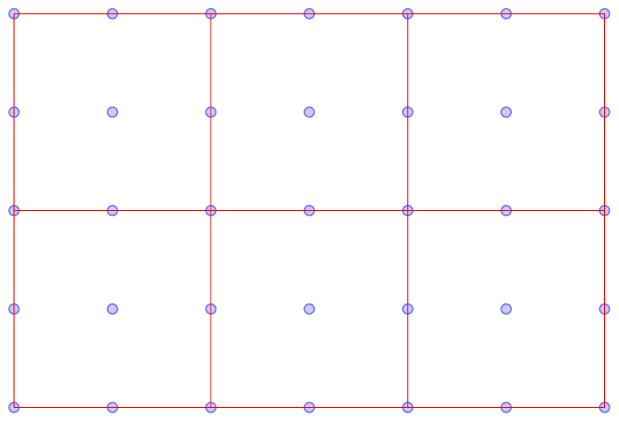} }
 \hspace{.6in}
 \subfigure[The corresponding finite difference grid]{\includegraphics[scale=0.8]{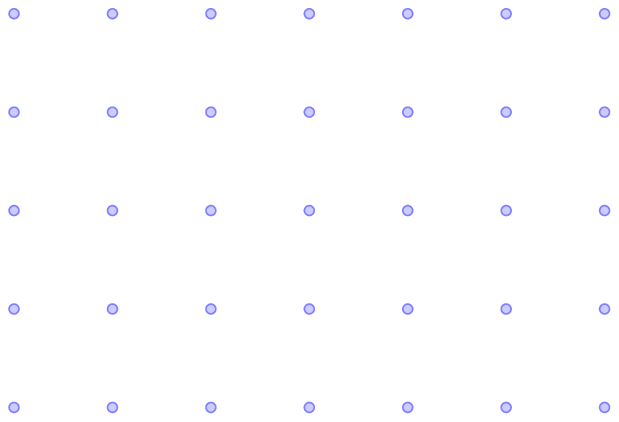}}
\caption{An illustration of Lagrangian $Q^2$ element and the $3\times3$ Gauss-Lobatto quadrature. }
\label{mesh}
 \end{figure}

 It is well known that many classical finite difference
schemes are exactly  finite element methods with specific  quadrature scheme, see \cite{ciarlet1991basic}.
 We will write scheme \eqref{intro-scheme2} as an exact finite difference type scheme in Section \ref{sec-fd} for $k=2$. Such a finite difference scheme not only provides an efficient way for assembling the stiffness matrix especially for a variable coefficient problem, but also with has advantages inherited from the variational formulation, such as symmetry of stiffness matrix and easiness of handling boundary conditions in high order schemes. This is the variational approach to construct a high order accurate finite difference scheme .

  Classical quadrature error estimates imply that standard finite element error estimates still hold for \eqref{intro-scheme2}, see \cite{ciarlet1972combined, ciarlet1991basic}. 
 The focus of this paper is to prove that the superconvergence of function values at Gauss-Lobatto points still holds. To be more specific, for Dirichlet type boundary conditions, we will show that \eqref{intro-scheme2} is a $(k+2)$-th order accurate finite difference scheme in the discrete 2-norm under suitable smoothness assumptions on the exact solution and the coefficients. 
 
 In this paper, the main motivation to study superconvergence is to use it for constructing  $(k+2)$-th order accurate finite difference schemes. For such a task, superconvergence points should define all degree of freedoms over the whole computational domain including boundary points. For  high order finite element methods, this seems possible only on  quite structured meshes such as
 rectangular meshes for a rectangular domain and equilateral triangles for a hexagonal domain, even though there are numerous superconvergence  results  for interior cells in unstructured meshes. 
 \subsection{Related work and difficulty in using standard tools}
 
 To illustrate our perspectives and difficulties, we focus on the case $k=2$ in the following. 
 For computing the bilinear form in the scheme \eqref{intro-scheme1}, another convenient implementation is to replace the smooth coefficient $a(x,y)$ by a piecewise $Q^2$ polynomial $a_I(x,y)$ obtained by interpolating $a(x,y)$ at the quadrature points in each cell shown in Figure \ref{mesh}. Then one can compute the integrals in the bilinear form exactly since the integrand is a polynomial. Superconvergence of function values for such an approximated coefficient scheme was proven in  \cite{li2019superconvergence} and the proof can be easily extended to higher order polynomials and three-dimensional cases. This result might seem surprising since interpolation error $a(x,y)-a_I(x,y)$ is of third order. On the other hand, all the tools used in \cite{li2019superconvergence}  are standard in the literature. 
 
From a practical point of view, \eqref{intro-scheme2}  is more interesting since it gives a genuine finite difference scheme. 
It is straightforward to use standard tools in the literature for showing superconvergence still holds for accurate enough quadrature. Even though the $3\times 3$ Gauss-Lobatto quadrature is fourth order accurate, the standard quadrature error estimates cannot be used directly to establish the fourth order accuracy of \eqref{intro-scheme2},
as will be explained in detail in Remark \ref{rmk-consistency} in Section \ref{sec-refinedconsistency}.

  We can also rewrite \eqref{intro-scheme2} for $k=2$ as a  finite difference scheme but its local truncation error is only second order as will be shown in Section \ref{sec-Laplacian}. The phenomenon that truncation errors have lower orders was named {\it supraconvergence} in the literature. The second order  truncation error makes it  difficult to establish the fourth order accuracy following any traditional finite difference analysis approaches. 
 
 To construct high order finite difference schemes from variational formulation, 
we can also consider finite element method with $P^2$ basis on a regular triangular mesh (two adjacent triangles form a rectangle) \cite{whiteman1975lagrangian}. Superconvergence of function values in $C^0$-$P^2$ finite element method at the three vertices and three edge centers can  be proven  \cite{chen2001structure, wahlbin2006superconvergence}. See also \cite{huang2008superconvergence}.  Even though the quadrature using only three edge centers is third order accurate,  error cancellations happen on two  adjacent triangles
forming a rectangle, thus fourth order accuracy of the corresponding finite difference scheme is still possible. However,  extensions to construct higher order finite difference schemes are much more difficult. 
  
\subsection{Contributions and organization of the paper}

The main contribution is to give the proof of the $(k+2)$-th order accuracy of  \eqref{intro-scheme2}, which is an easy construction of high order finite difference schemes for variable coefficient problems.
An important step is to obtain desired sharp quadrature estimate for the bilinear form, for which it is necessary to count in quadrature error cancellations between neighboring cells.  Conventional quadrature estimating tools 
such as the Bramble-Hilbert Lemma only give the sharp estimate on each cell thus cannot be used directly. 
A key technique in this paper is to apply the Bramble-Hilbert Lemma after integration by parts on proper interpolation polynomials to allow error cancellations.

 The paper is organized as follows. In Section \ref{sec-preliminaries}, we introduce our notations and assumptions. 
In Section \ref{sec-consistency}, standard quadrature estimates are reviewed. 
  Superconvergence of bilinear forms with quadrature is shown in Section \ref{sec-bilinear}. Then we prove the main result for homogeneous Dirichlet boundary conditions in Section \ref{sec-main} and for nonhomogeneous Dirichlet boundary conditions in Section \ref{sec-nonhomogeneous-dirichlet}.
 Section \ref{sec-fd} provides a simple finite difference implementation of \eqref{intro-scheme2}. 
 Section \ref{sec-test} contains numerical tests. Concluding remarks are given in Section \ref{sec-conclusion}.
  
\section{Notations and assumptions}
\label{sec-preliminaries}
\subsection{Notations and basic tools}

We will use the same notations as in \cite{li2019superconvergence}: 
\begin{itemize}
 \item We only consider a rectangular domain $\Omega=(0,1)\times(0,1)$ with its boundary denoted as $\partial \Omega$. 
 \item Only for convenience, we assume $\Omega_h$ is an uniform rectangular mesh for $\bar\Omega$ and $e=[x_e-h,x_e+h]\times [y_e-h,y_e+h]$ denotes any cell in $\Omega_h$ with  cell center $(x_e,y_e)$. {\bf The assumption of an uniform mesh is not essential to the discussion of superconvergence}.
  \item $Q^k(e)=\left\{p(x,y)=\sum\limits_{i=0}^k\sum\limits_{j=0}^k p_{ij} x^iy^j, (x,y)\in e\right\}$ is the set of 
 tensor product of polynomials of degree $k$ on a cell $e$.
\item $V^h=\{p(x,y)\in C^0(\Omega_h): p|_e \in Q^{k}(e),\quad \forall e\in \Omega_h\}$ denotes the continuous piecewise $Q^{k}$ finite element space on $\Omega_h$.
\item $V^h_0=\{v_h\in V^h: v_h=0 \quad\mbox{on}\quad \partial \Omega \}.$
\item The norm and seminorms for $W^{k,p}(\Omega)$ and $1\leq p<+\infty$, with standard modification for $p=+\infty$:
$$
 \|u\|_{k,p,\Omega}=\left(\sum\limits_{i+j\leq  k}\iint_{\Omega}|\partial_x^i\partial_y^ju(x,y)|^pdxdy\right)^{1/p},
 $$
$$
 |u|_{k,p,\Omega}=\left(\sum\limits_{i+j= k}\iint_{\Omega}|\partial_x^i\partial_y^ju(x,y)|^pdxdy\right)^{1/p},
 $$
$$
 [u]_{k,p,\Omega}=\left(\iint_{\Omega}|\partial_x^k u(x,y)|^pdxdy+\iint_{\Omega}|\partial_y^k u(x,y)|^p dxdy\right)^{1/p}.
 $$
 Notice that $[u]_{k+1,p,\Omega}=0$ if $u$ is a $Q^k$ polynomial.
 \item For simplicity, sometimes we may use $\|u\|_{k,\Omega}$, $|u|_{k,\Omega}$ and $[u]_{k,\Omega}$ denote 
 norm and seminorms for $H^k(\Omega)=W^{k,2}(\Omega)$.
 \item When there is no confusion, $\Omega$ may be dropped in the norm and seminorms, e.g., $\|u\|_k=\|u\|_{k,2,\Omega}$.
 \item For any $v_h\in V^h$, $1\leq p<+\infty $ and $k\geq 1$, $$\|v_h\|_{k,p,\Omega}:= \left(\sum_e\|v_h\|_{k,p,e}^p\right)^{\frac1p}, \quad 
 |v_h|_{k,p,\Omega}:= \left(\sum_e|v_h|_{k,p,e}^p\right)^{\frac1p},\quad  [v_h]_{k,p,\Omega}:= \left(\sum_e[v_h]_{k,p,e}^p\right)^{\frac1p}.$$
\item Let $Z_{0,e}$ denote the set of $(k+1)\times(k+1)$ Gauss-Lobatto points on a cell $e$.
\item $Z_0=\bigcup_e Z_{0,e}$ denotes all Gauss-Lobatto points in the mesh $\Omega_h$.
\item 
Let $\|u\|_{2,Z_0}$ and $\|u\|_{\infty,Z_0}$
denote the discrete 2-norm and the maximum norm over $Z_0$ respectively:
\[\|u\|_{2,Z_0}=\left[h^2\sum_{(x,y)\in Z_0} |u(x,y)|^2\right]^{\frac12},\quad \|u\|_{\infty,Z_0}=\max_{(x,y)\in Z_0} |u(x,y)|.\]
\item For a continuous function $f(x,y)$, let $f_I(x,y)$ denote its piecewise $Q^k$ Lagrange interpolant at $Z_{0,e}$ on each cell $e$, i.e., $f_I\in V^h$ satisfies:
\[f(x,y)=f_I(x,y), \quad \forall (x,y)\in Z_0. \]
\item $P^k(t)$ denotes the polynomial of degree $k$ of variable $t$. 
\item 
$(f,v)_e$ denotes the inner product in  $L^2(e)$ and $(f,v)$ denotes the inner product in $L^2(\Omega)$:
\[(f,v)_e=\iint_{e} fv\, dxdy,\quad (f,v)=\iint_{\Omega} fv\, dxdy=\sum_e (f,v)_e.\]
\item $\langle f,v\rangle_{e,h}$ denotes the approximation to $(f,v)_e$ by using $(k+1)\times (k+1)$-point Gauss Lobatto quadrature for integration over cell $e$. 
\item $\langle f,v\rangle_h$ denotes the approximation to $(f,v)$ by using $(k+1)\times(k+1)$-point Gauss Lobatto
quadrature for integration over each cell $e$. 
\item $\hat K=[-1,1]\times [-1,1]$ denotes a reference cell.
\item For $f(x,y)$ defined on $e$, consider $\hat f(s, t)=f(sh+ x_e,t h+ y_e)$  defined on $\hat K$. 
Let $\hat f_I$ denote the $Q^k$ Lagrange interpolation of $\hat f$ at the $(k+1) \times (k+1)$  Gauss Lobatto quadrature points on $\hat K$. 
\item $(\hat f,\hat v)_{\hat K}=\iint_{\hat K} \hat f\hat v\, dsdt.$
\item $\langle \hat f,\hat v\rangle_{\hat K}$
   denotes the approximation to $(\hat f,\hat v)_{\hat K}$ by using $(k+1) \times (k+1)$-point Gauss-Lobatto quadrature.
 \item On the reference cell $\hat K$, for convenience we use the superscript $h$ over the $ds$ or $dt$ to denote we use $(k+1)$-point Gauss-Lobatto quadrature on the corresponding variable. For example, 
$$\iint_{\hat K} \hat f d^hsdt=\int_{-1}^{1}[w_1\hat f(-1,t)+w_{k+1}\hat f(1,t)+ \sum_{i=2}^k w_i\hat f(x_i,t)] dt.$$ Since $(\hat f\hat v)_I$ coincides with $\hat f\hat v$ at the quadrature points, 
we have $$\iint_{\hat K} (\hat f\hat v)_I dxdy=\iint_{\hat K} (\hat f\hat v)_I d^hxd^hy=\iint_{\hat K} \hat f\hat v d^hxd^hy=\langle \hat f,\hat v\rangle_{\hat K}.$$
\end{itemize}
  
  \bigskip
  
The following are commonly used
  tools and facts:
  \begin{itemize}
 \item For two-dimensional problems, 
 $$h^{k-2/p}|v|_{k,p,e}=|\hat v|_{k,p,\hat K},\quad  h^{k-2/p}[v]_{k,p,e}=[\hat v]_{k,p,\hat K}, \quad 1\leq p\leq \infty.$$
 \item Inverse estimates for polynomials:
 \begin{equation}
 \|v_h\|_{k+1, e}\leq C h^{-1} \|v_h\|_{k, e},\quad  \forall v_h \in V^h, k\geq 0. 
\label{inverseestimate}
 \end{equation}

 \item Sobolev's embedding in two and three dimensions: $H^{2}(\hat K)\hookrightarrow C^0(\hat K)$. 
 \item The embedding implies  
  $$\|\hat  f\|_{0,\infty,\hat K}\leq C \|\hat  f\|_{k,2, \hat K},\quad \forall \hat f\in H^{k}(\hat K), k\geq 2,$$
 $$\|\hat  f\|_{1,\infty,\hat K}\leq C \|\hat  f\|_{k+1,2, \hat K}, \quad \forall\hat f\in H^{k+1}(\hat K), k\geq 2.$$
   \item Cauchy-Schwarz inequalities in two dimensions:
  \[\sum_e \|u\|_{k,e}\|v\|_{k,e}\leq \left(\sum_e \|u\|^2_{k,e}\right)^{\frac12}\left(\sum_e \|v\|^2_{k,e}\right)^{\frac12}, \quad
  \|u\|_{k,1,e}=\mathcal O(h) \|u\|_{k,2,e}.\]
  \item Poincar\'{e} inequality: let $\bar u$ be the average of $u\in H^1(\Omega)$ on $\Omega$, then 
  \[ |u-\bar u|_{0,p,\Omega}\leq C |\nabla u|_{0,p,\Omega}, \quad p\geq 1.\]
  If $\bar u$ is the average of $u\in H^1(e)$ on a cell $e$, we have 
    \[ |u-\bar u|_{0,p,e}\leq C h |\nabla u|_{0,p,e}, \quad p\geq 1.\]
 \item 
For $k\geq 2$, the $(k+1)\times(k+1)$ Gauss-Lobatto quadrature is exact for 
integration of polynomials of degree $2k-1\geq k+1$ on $\hat K$.
\item  Define the projection operator $\hat{\Pi}_1: \hat u \in L^1(\hat K)\rightarrow \hat \Pi_1\hat u\in Q^1(\hat K)$ by 
 \begin{equation}
\iint_{\hat K} (\hat{\Pi}_1 \hat{u} ) w dsdt= \iint_{\hat K} \hat{u} w dsdt,\forall w\in Q^1(\hat K).
\label{projection1}
 \end{equation}
 Notice that all degree of freedoms of $\hat{\Pi}_1 \hat{u}$ can be represented as a linear combination of $\iint_{\hat K} \hat u(s,t) p(s,t)dsdt$ for $p(s,t)=1,s,t,st$, 
 thus $\hat{\Pi}_1$ is a continuous linear mapping from $L^2(\hat K)$ to $H^1(\hat K)$ (or $H^2(\hat K)$) by Cauchy-Schwarz inequality $|\iint_{\hat K} \hat u(s,t) \hat p(s,t)dsdt|\leq \|\hat u\|_{0,2,\hat K}\|\hat p\|_{0,2,\hat K}\leq C \|\hat u\|_{0,2,\hat K}$.
  \end{itemize}

%

\subsection{Coercivity and elliptic regularity}
We consider the elliptic variational problem  of finding $u\in H_0^1(\Omega)$ to satisfy
\begin{equation}
A(u,v): =\iint_\Omega (\nabla v^T \mathbf a \nabla u +\mathbf b\nabla u v + c u v) \,dx dy =(f,v), \forall v\in H^1_0(\Omega),
\label{bilinearform}
\end{equation}
where $\mathbf a=\begin{pmatrix}
               a^{11} & a^{12}\\
               a^{21} & a^{22}
              \end{pmatrix}
$ is positive definite and $\mathbf b=[b^1 \quad b^2]$. 
Assume the coefficients $\mathbf a$, $\mathbf b$ and $c$ are smooth with  uniform upper bounds, thus $A(u,v)\leq C\|u\|_1\|v\|_1$ for any $u, v\in H^1_0(\Omega)$. We denote $\lambda_{\mathbf a}$ as the smallest eigenvalues of $\mathbf a$.
Assume $\lambda_{\mathbf a}$ has a positive lower bound and $ \nabla \cdot \mathbf b\leq 2c $, so that coercivity of the bilinear form can be easily achieved. 
Since
$$(\mathbf b\cdot \nabla u,v)=\int_{\partial\Omega} uv \mathbf b\cdot \mathbf n ds-(\nabla\cdot (v\mathbf b), u)=\int_{\partial\Omega} uv \mathbf b\cdot \mathbf n ds-(\mathbf b\cdot\nabla v, u)-(v\nabla \cdot \mathbf b, u), $$
we have
\begin{equation}
2(\mathbf b\cdot \nabla v,v)+2(c v, v)=\int_{\partial\Omega} v^2 \mathbf b\cdot \mathbf n ds+((2c-\nabla \cdot \mathbf b) v, v)\geq 0,\quad\forall v\in H^1_0(\Omega).
\label{ibp1}
\end{equation}
By the equivalence of two norms $|\cdot|_{1}$ and $\|\cdot\|_{1}$ for the space $H^1_0(\Omega)$ (see \cite{ciarlet1991basic}),  
we conclude that the bilinear form 
$A(u,v)=(\mathbf a \nabla u, \nabla v)+(\mathbf b\cdot \nabla u,v)+(c u, v)$
satisfies coercivity $A(v,v)\geq C \|v\|_1$ for any $v\in H^1_0(\Omega)$. 

The coercivity can also be achieved if we assume $| \mathbf b|< 4\lambda_{\mathbf a}c $.
By Young's inequality
$$|( \mathbf b\cdot \nabla v, v )| \leq \iint_{\Omega} \frac{|\mathbf b \cdot \nabla v|^2}{4c} + c |v|^2dxdy \leq 
\left ( \frac{|\mathbf b|^2}{4c} \nabla v, \nabla v\right ) + ( c v, v),$$
we have
\begin{equation}
A(v,v) \geq (\mathbf a\cdot \nabla v, \nabla v) +(c v, v) -|( \mathbf b\cdot \nabla v, v )| \geq \left( (\lambda_{\mathbf a}-\frac{|\mathbf b|^2}{4c}) \nabla v, \nabla v\right ) >0,\quad\forall v\in H^1_0(\Omega).
\label{ibp2}
\end{equation}

We need to make an additional assumption for \eqref{bilinearform}: 
the elliptic regularity holds for the dual problem.  
Let $A^*$ be the dual operator of $A$, i.e., $A^*(u,v)=A(v,u)$. We  assume the elliptic regularity $\|w\|_2\leq C\|f\|_0$ holds for the exact dual problem of finding $w\in H^1_0(\Omega)$ satisfying
$A^*(w,v)=(f,v),\quad \forall v\in H_0^1(\Omega)$.
 See \cite{savare1998regularity, grisvard2011elliptic} for the elliptic regularity with Lipschitz continuous coefficients on a Lipschitz domain.

\section{Quadrature error estimates}
\label{sec-consistency}
In the following, we will use  $\hat{\quad}$  for a function to  emphasize the function is defined on or  transformed to the reference cell  $\hat K =[-1,1]\times [-1,1]$ from a mesh cell.
\subsection{Standard estimates}
The Bramble-Hilbert Lemma for $Q^k$ polynomials can be stated as follows,  see Exercise 3.1.1 and Theorem 4.1.3 in \cite{ciarlet1978finite}:
\begin{theorem}
\label{bh-lemma}
If a continuous linear mapping  $\hat\Pi: H^{k+1}(\hat K)\rightarrow H^{k+1}(\hat K)$
satisfies $\hat\Pi \hat v=\hat v$ for any $\hat v\in Q^k(\hat K)$, then 
\begin{equation}
\|\hat u-\hat \Pi\hat u\|_{k+1,\hat K}\leq C [\hat u]_{k+1, \hat K}, \quad \forall\hat u\in H^{k+1}(\hat K).
\label{bh1}
\end{equation}
Thus if $l(\cdot)$ is a continuous linear form on the space $H^{k+1}(\hat K)$ satisfying
$l(\hat v)=0,\forall \hat v\in Q^k(\hat K),$
then \[|l(\hat u)|\leq C \|l\|'_{k+1, \hat K} [\hat u]_{k+1,\hat K},\quad \forall\hat u\in H^{k+1}(\hat K),\]
where $ \|l\|'_{k+1, \hat K}$ is the norm in the dual space of $H^{k+1}(\hat K)$.
\end{theorem}

By applying Bramble-Hilbert Lemma, we have the following standard quadrature estimates. See \cite{li2019superconvergence}  for the detailed proof. 
\begin{theorem}
For  a sufficiently smooth function $a(x,y)$, let $m$ is an integer satisfying $k\leq  m\leq 2k$, 
we have 
\label{quaderror-theorem}
\[  \iint_e a(x,y)dxdy-  \iint_e a_I(x,y)dxdy=\mathcal O(h^{m+1})[a]_{m,e}=\mathcal O(h^{m+2})[a]_{m,\infty,e}.\]

\end{theorem}
\begin{theorem}
\label{rhs-estimate}
 If $f\in H^{k+2}(\Omega)$, $(f,v_h)-\langle f,v_h\rangle_h =\mathcal O(h^{k+2}) \|f\|_{k+2} \|v_h\|_2,\quad\forall v_h\in V^h.$
\end{theorem}

\begin{remark}
By the theorems above, on the reference cell $\hat K$,  we have 
\begin{equation}\label{quaderror-theorem-ref}
\iint_{\hat K} \hat a(s,t)-\hat a_I(s,t)dsdt\leq C [\hat a]_{k+2,\hat K}\leq C[\hat a]_{k+2,\infty,\hat K},
\end{equation}
and 
\begin{equation}\label{remark-interp-error}
\|\hat a-\hat a_I\|_{k+1,\hat K}\leq C [\hat a]_{k+1,\hat K}.
\end{equation}
\end{remark}

The following two results are also standard estimates obtained by applying the Bramble-Hilbert Lemma. 
\begin{lemma}
\label{lemma-quaderror-2norm}
 If $f\in H^{2}(\Omega)$ or $f\in V^h$, we have $(f,v_h)-\langle f,v_h\rangle_h =\mathcal O(h^2) |f|_{2} \|v_h\|_0,\quad\forall v_h\in V^h.$
\end{lemma}
\begin{proof}
For simplicity, we ignore the subscript in $v_h$.
Let $E(f)$ denote the quadrature error for integrating $f(x,y)$ on $e$.
Let $\hat E(\hat f)$ denote the quadrature error for integrating $\hat f(s,t)=f(x_e+sh,y_e+th)$ on the 
reference cell $\hat K$.
 Due to the embedding $H^{2}(\hat K)\hookrightarrow C^0(\hat K)$, we have
 \begin{align*}
|\hat E(\hat f \hat v)|\leq C |\hat f \hat v|_{0,\infty,\hat K}\leq C |\hat f |_{0,\infty,\hat K}|\hat v|_{0,\infty,\hat K}\leq C\|\hat f\|_{2,\hat K} \|\hat v\|_{0,\hat K}.
 \end{align*}
 Thus the mapping $\hat f\rightarrow E(\hat f \hat v)$ is a continuous linear form on $H^2(\hat K)$ and its norm is bounded by $C\|\hat v\|_{0,\hat K}$.
If 
$\hat f \in Q^1(\hat K)$, then we have $\hat E(\hat f \hat v)=0$.
By the Bramble-Hilbert Lemma Theorem \ref{bh-lemma} on this continuous linear form, we get  
$$|\hat E(\hat f \hat v)| \leq C [\hat f]_{2,\hat K} \|\hat v\|_{0,\hat K} .$$ 
So on a cell $e$, we get
\begin{equation}\label{consisterr-cell}
E(fv)=h^2\hat E (\hat f\hat v)\leq Ch^2[\hat f]_{2,\hat K} \|\hat v\|_{0,\hat K}\leq C h^2|f|_{2, e} \| v\|_{0, e}.
\end{equation}
 Summing over all elements and use Cauchy-Schwarz inequality, we get the desired result. 
\end{proof}

\begin{theorem}\label{a-ah-1norm}
Assume all coefficients of \eqref{bilinearform} are in $W^{2,\infty}(\Omega)$. We have 
\[ A(z_h,v_h)-A_h(z_h,v_h)=\mathcal O(h) \|v_h\|_2 \|z_h\|_1, \quad\forall v_h,z_h\in V^h.\]
\end{theorem}
\begin{proof}
By setting $f=a^{11} (v_h)_x$ in \eqref{consisterr-cell}, we get 
\begin{align*}
& |(a^{11} (z_h)_x, (v_h)_x)-\langle a^{11} (z_h)_x, (v_h)_x\rangle_h|\leq C h^2 \|a^{11} (v_h)_x\|_2 \|(z_h)_x\|_0\\
 \leq & C h^2 \|a^{11}\|_{2,\infty} \|v_h\|_3 |z_h|_1 \leq  C h \|a^{11}\|_{2,\infty} \|v_h\|_2 |z_h|_1,
\end{align*}
where the inverse estimate \eqref{inverseestimate} is used in the last inequality.
Similarly, we have
\begin{align*}
(a^{12} (z_h)_x, (v_h)_y)-\langle a^{12} (z_h)_x, (v_h)_y\rangle_h&= C h \|a^{12}\|_{2,\infty}\|v_h\|_2 |z_h|_1,\\
(a^{22} (z_h)_y, (v_h)_y)-\langle a^{22} (z_h)_y, (v_h)_y\rangle_h&= C h \|a^{22}\|_{2,\infty}\|v_h\|_2 |z_h|_1,\\
(b^1 (z_h)_x, v_h)-\langle b^1 (z_h)_x, v_h\rangle_h&= C h \|b^1\|_{2,\infty} \|v_h\|_2 |z_h|_0,\\
(b^2 (z_h)_y, v_h)-\langle b^2 (z_h)_y, v_h\rangle_h&= C h \|b^2\|_{2,\infty} \|v_h\|_2 |z_h|_0,\\
(c z_h, v_h)-\langle c z_h, v_h\rangle_h&= C h \|c\|_{2,\infty} \|v_h\|_1 |z_h|_0,
\end{align*}
which implies
\[A(z_h,v_h)-A_h(z_h,v_h)=\mathcal O(h) \|v_h\|_2 \|z_h\|_1.\]
\end{proof}

 \subsection{A refined consistency error}
 \label{sec-refinedconsistency}
 In this subsection, we will show how to establish the desired consistency error 
 estimate for smooth enough coefficients: 
 \[
 A(u, v_h)-A_h(u,v_h)=\begin{cases}
                         \mathcal O(h^{k+2})\|u\|_{k+3}\|v_h\|_2,\quad \forall v_h\in V^h_0\\
                         \mathcal O(h^{k+\frac{3}{2}})\|u\|_{k+3}\|v_h\|_2,\quad \forall v_h\in V^h
                        \end{cases}.
\]
 
\begin{theorem}
\label{a-ah}
Assume $a(x,y)\in W^{k+2,\infty}(\Omega)$, $u\in H^{k+3}(\Omega)$, then
\begin{subnumcases}
{(a \partial_x u,\partial_x v_h)-\langle a\partial_x u, \partial_x v_h\rangle_h=}
\mathcal O(h^{k+2})\|a\|_{k+2,\infty}\|u\|_{k+3}\|v_h\|_2,\quad \forall v_h\in V^h_0,\label{a-ah-auxvxv0}\\
\mathcal O(h^{k+\frac32})\|a\|_{k+2,\infty}\|u\|_{k+3}\|v_h\|_2,\quad \forall v_h\in V^h,\label{a-ah-auxvxv}
\end{subnumcases}
\begin{subnumcases}
{(a\partial_x u,\partial_y v_h)-\langle a\partial_x u, \partial_y v_h\rangle_h=}
\mathcal O(h^{k+2})\|a\|_{k+2,\infty}\|u\|_{k+3}\|v_h\|_2,\quad \forall v_h\in V^h_0,\label{a-ah-auxvyv0}\\
\mathcal O(h^{k+\frac32})\|a\|_{k+2,\infty}\|u\|_{k+3}\|v_h\|_2,\quad \forall v_h\in V^h,\label{a-ah-auxvyv}
\end{subnumcases}
\begin{equation}
(a\partial_x u,v_h)-\langle a\partial_x u, v_h\rangle_h=\mathcal O(h^{k+2})\|a\|_{k+2,\infty}\|u\|_{k+3}\|v_h\|_2,\quad \forall v_h\in V^h_0,\label{a-ah-buxv}
\end{equation}
\begin{equation}
(au,v_h)-\langle au, v_h\rangle_h=\mathcal O(h^{k+2})\|a\|_{k+2,\infty}\|u\|_{k+2}\|v_h\|_2,\quad \forall v_h\in V^h_0.\label{a-ah-cuv}
\end{equation}
\end{theorem}
\begin{remark}
 \label{rmk-consistency}
 We emphasize that Theorem \ref{a-ah} cannot be proven by applying the Bramble-Hilbert Lemma directly. Consider the constant coefficient case $a(x,y)\equiv 1$ and $k=2$ as an example, 
 $$(\partial_x u, \partial_x v_h)-\langle\partial_x u, \partial_x v_h\rangle_h=\sum_e  \left(\iint_e u_x (v_h)_x dxdy -\iint_e u_x(v_h)_x d^hxd^hy\right).$$
Since the $3 \times 3$ Gauss-Lobatto quadrature is exact for integrating $Q^3$ polynomials, 
by Theorem \ref{bh-lemma} we have 
 \[\left|\iint_e u_x (v_h)_x dxdy -\iint_e u_x(v_h)_x d^hxd^hy\right|=
 \left|\iint_{\hat K}  \hat u_s (\hat v_h)_s dsdt -\iint_{\hat K} \hat u_s(\hat v_h)_s d^hsd^ht\right|\leq C [\hat u_s (\hat v_h)_s]_{4,\hat K}.\]
 Notice that  $\hat v_h$ is $Q^2$ thus $(\hat v_h)_{stt}$ does not vanish and 
 $[(\hat v_h)_s]_{4,\hat K}\leq C |\hat v_h|_{3,\hat K}$. 
  So by Bramble-Hilbert Lemma for $Q^k$ polynomials, we can only get 
 \[\iint_e u_x (v_h)_x dxdy -\iint_e u_x(v_h)_x d^hxd^hy=\mathcal O(h^4) \|u\|_{5,e} \|v_h\|_{3,e}.\]
Thus by Cauchy-Schwarz inequality after summing over $e$, we only have 
 \[(\partial_x u, \partial_x v_h)-\langle\partial_x u, \partial_x v_h\rangle_h=\mathcal O(h^4) \|u\|_5\|v_h\|_3.\]
  \end{remark}

 In order to get the desired estimate involving only the $H^2$-norm of $v_h$, we will take advantage of error cancellations between neighboring cells through integration by parts.

 \begin{proof}
For simplicity, we ignore the subscript $_h$ of $v_h$ in this proof and all the following $v$ are in $V^h$ which are $Q^k$ polynomials in each cell.  
First, by Theorem \ref{rhs-estimate}, we easily obtain \eqref{a-ah-buxv} and \eqref{a-ah-cuv}: 
\begin{equation*}
(au_x,v)-\langle au_x, v\rangle_h=\mathcal O(h^{k+2})\|au_x\|_{k+2}\|v\|_{2}=\mathcal O(h^{k+2})\|a\|_{k+2,\infty}\|u\|_{k+3}\|v\|_{2},
\end{equation*}
\begin{equation*}
(au,v)-\langle au, v\rangle_h=\mathcal O(h^{k+2})\|au\|_{k+2}\|v\|_{2}=\mathcal O(h^{k+2})\|a\|_{k+2,\infty}\|u\|_{k+2}\|v\|_{2}.
\end{equation*}
We will only discuss $(au_x, v_x)-\langle au_x, v_x\rangle_h$ and the same discussion also applies to derive \eqref{a-ah-auxvyv0} and \eqref{a-ah-auxvyv}.

Since we have
\begin{align*}
 &(au_x, v_x)-\langle au_x, v_x\rangle_h=\sum_e \left( \iint_e au_x v_x dxdy -\iint_e au_xv_x d^hxd^hy\right)\\
= & \sum_e \left( \iint_{\hat K} \hat a \hat u_s \hat v_s dsdt -\iint_{\hat K} \hat a \hat u_s \hat v_s d^hsd^ht \right)
=\sum_e  \left( \iint_{\hat K} \hat a\hat u_s \hat v_s dsdt -\iint_{\hat K} (\hat a\hat u_s)_I\hat v_s d^hsd^ht \right),
 \end{align*}
where we use the fact $\hat a\hat u_s \hat v_s = (\hat a\hat u_s)_I\hat v_s$ on the Gauss-Lobatto quadrature points.
For fixed $t$, $(\hat a\hat u_s)_I\hat v_s$ is a polynomial of degree $2k-1$ w.r.t. variable $s$, thus the $(k+1)$-point Gauss-Lobatto quadrature is exact for its $s$-integration, i.e.,
$$\iint_{\hat K} (\hat a\hat u_s)_I\hat v_s d^hsd^ht=\iint_{\hat K} (\hat a\hat u_s)_I\hat v_s dsd^ht .$$
To estimate the quadrature error we introduce some intermediate values then do interpretation by parts,
\begin{align}
&\iint_{\hat K} \hat a\hat u_s \hat v_s dsdt -\iint_{\hat K} (\hat a\hat u_s)_I\hat v_s d^hsd^ht\\
=&\iint_{\hat K} \hat a\hat u_s \hat v_s dsdt -\iint_{\hat K} (\hat a\hat u_s)_I\hat v_s dsdt + \iint_{\hat K} (\hat a\hat u_s)_I\hat v_s dsdt-\iint_{\hat K} (\hat a\hat u_s)_I\hat v_s dsd^ht\\
= & \iint_{\hat K} \left[\hat a\hat u_s - (\hat a\hat u_s)_I \right] \hat v_s dsdt + \left( \iint_{\hat K} \left[(\hat a\hat u_s)_I\right]_s\hat v dsd^ht -\iint_{\hat K} \left[(\hat a\hat u_s)_I\right]_s\hat v dsdt \right) \\
& +\left( \left.\int_{-1}^1 (\hat a\hat u_s)_I\hat v dt \right|^{s=1}_{s=-1} - \left. \int_{-1}^1 (\hat a\hat u_s)_I\hat v d^ht \right|^{s=1}_{s=-1} \right) = I + II + III.\label{1dquadontint}
\end{align}

For the first term in \eqref{1dquadontint}, let $\overline{\hat v_s}$ be the cell average of $\hat v_s$ on $\hat K$, then  
\begin{align*}
I = \iint_{\hat K} \left(\hat a\hat u_s -(\hat a\hat u_s)_I\right)\overline{\hat v_s} dsdt+\iint_{\hat K} \left(\hat a\hat u_s -(\hat a\hat u_s)_I\right)(\hat v_s-\overline{\hat v_s})dsdt.
\end{align*}
 By \eqref{quaderror-theorem-ref} we have
\[
 \left|\iint_{\hat K} \left(\hat a\hat u_s -(\hat a\hat u_s)_I\right)\overline{\hat v_s} dsdt  \right| \leq C [\hat a\hat u_s]_{k+2,\hat K}\left |\overline{\hat v_s}\right | =\mathcal O(h^{k+2})\|\hat a\|_{k+2,\infty,e}\|\hat u\|_{k+3,e}\|\hat v\|_{1,e}.
\]
By Cauchy-Schwarz inequality, the Bramble-Hilbert Lemma on interpolation error and Poincar\'{e} inequality, we have 
\begin{align*}
& \left| \iint_{\hat K} \left(\hat a\hat u_s -(\hat a\hat u_s)_I\right)(\hat v_s-\overline{\hat v_s})dsdt \right| \leq|\hat a\hat u_s -(\hat a\hat u_s)_I|_{0,\hat K}|\hat v_s-\overline{\hat v_s}|_{0,\hat K} \\
\leq &  C[\hat a\hat u_s ]_{k+1,\hat K}|\hat v|_{2,\hat K} =\mathcal O(h^{k+2}) \|a\|_{k+1,\infty, e}\|u\|_{k+2, e}\|v\|_{2,e}.
\end{align*}
Thus we have 
\[I =\mathcal O(h^{k+2}) \|a\|_{k+2,\infty,e}\|u\|_{k+3,e}\|v\|_{2,e}.
\]
 For the second term in \eqref{1dquadontint}, we can estimate it the same way as in the proof of Theorem 2.4. in \cite{li2019superconvergence}.
For each $\hat v \in Q^k(\hat K)$ we can define a linear form as 
$$\hat E_{\hat v}(\hat f) = \iint_{\hat K}(\hat F_I)_s \hat v dsdt -\iint_{\hat K} (\hat F_I)_s\hat v dsd^ht,$$
where $\hat F$ is an antiderivative of $\hat f$ w.r.t. variable $s$. Due to the linearity of interpolation operator and differentiating operation, $\hat E_{\hat v}$ is well defined. 
By the embedding $H^{2}(\hat K)\hookrightarrow C^0(\hat K)$, we have
$$
\hat E_{\hat v}(\hat f) \leq C \|\hat F\|_{0,\infty,\hat K}\| \hat v\|_{0,\infty,\hat K} \leq C \|\hat f\|_{0,\infty,\hat K}\| \hat v\|_{0,\infty,\hat K} \leq C \|\hat f\|_{2,\hat K} \| \hat v\|_{0,\hat K} \leq C \|\hat f\|_{k,\hat K} \| \hat v\|_{0,\hat K},
$$
which implies that the mapping $\hat E_{\hat v}$ is a continuous linear form on $H^k(\hat K)$.
With projection $\Pi_1$ defined in \eqref{projection1}, we have 
$$
\hat E_{\hat v}(\hat f) = \hat E_{\hat v - \Pi_1\hat v} (\hat f)+ \hat E_{ \Pi_1\hat v}(\hat f),\quad  \forall \hat v \in Q^k(\hat K).
$$ 
Since $ Q^{k-1}(\hat K) \subset \ker \hat E_{\hat v - \Pi_1\hat v}$, thus by the Bramble-Hilbert Lemma, 
$$
\hat E_{\hat v - \Pi_1\hat v} (\hat f)  \leq C[f]_{k,\hat K}\|\hat v- \Pi_1\hat v\|_{0,\hat K}\leq C[f]_{k,\hat K}|\hat v|_{2,\hat K},  
$$
 and we also have
 $$
\hat E_{\Pi_1\hat v} (\hat f) =\iint_{\hat K}(\hat F_I)_s \Pi_1\hat v dsdt -\iint_{\hat K} (\hat F_I)_s\Pi_1\hat v dsd^ht=0. 
$$
Thus we have 
 \begin{align*}
& \iint_{\hat K} \left[(\hat a\hat u_s)_I\right]_s\hat v dsd^ht -\iint_{\hat K} \left[(\hat a\hat u_s)_I\right]_s\hat v dsdt  
= -\hat E_{\hat v}((\hat a\hat u_s)_s) = -\hat E_{\hat v - \Pi_1\hat v} ((\hat a\hat u_s)_s) \\
\leq & C [(\hat a\hat u_s)_s]_{k,\hat K}|\hat v_h|_{2,\hat K} \leq C |\hat a\hat u_s|_{k+1,\hat K}|\hat v|_{2,\hat K} 
= \mathcal O(h^{k+2})\|a\|_{k+1,\infty, e}\|u\|_{k+2,e}|v|_{2,e} 
 \end{align*}

Now we only need to discuss the line integral term. Let $L_2$ and $L_4$ denote the left and right boundary of $\Omega$ and let $l^e_2$ and $l^e_4$ denote the left and right edge of element $e$ or $l^{\hat K}_2$ and $l^{\hat K}_4$ for $\hat K$. Since $(\hat a\hat u_s)_I\hat v$ mapped back to $e$ will be $\frac{1}{h}(au_x)_Iv$ which is continuous across $l^e_2$ and $l^e_4$, after summing over all elements $e$, the line integrals along the inner edges are canceled  out and only the line integrals on $L_2$ and $L_4$ remain. 

For a cell $e$ adjacent to $L_2$, consider its reference cell $\hat K$, and define a linear form  $\hat E(\hat f ) = \int_{-1}^1 \hat f(-1,t) dt - \int_{-1}^1 \hat f(-1,t)d^ht$, then we  have
$$\hat E(\hat f \hat v)  \leq C |\hat f|_{0,\infty, l^{\hat K}_2} |\hat v|_{0,\infty,l^{\hat K}_2} \leq C\|\hat f\|_{2,l^{\hat K}_2}\|\hat v\|_{0,l^{\hat K}_2},$$ which means that the mapping $\hat f \rightarrow \hat E(\hat f \hat v)$ is continuous with operator norm less than $C\|\hat v\|_{0,l^{\hat K}_2}$ for some $C$. Clearly we have 
\begin{align*}
\hat E(\hat f \hat v) = \hat E(\hat f \Pi_1 \hat v) + \hat E(\hat f( \hat v-\Pi_1 \hat v)).
\end{align*}

By   the Bramble-Hilbert Lemma \eqref{bh-lemma} we get 
\begin{align*}
 & \hat E((\hat a\hat u_s)_I(\hat v-\Pi_1 \hat v)) \leq C[(\hat a\hat u_s)_I]_{k,l^{\hat K}_2}[\hat v]_{2,l^{\hat K}_2} \leq C(|\hat a\hat u_s-(\hat a\hat u_s)_I|_{k,l^{\hat K}_2}+|\hat a\hat u_s|_{k,l^{\hat K}_2})[\hat v]_{2,l^{\hat K}_2}\\
 \leq & (|\hat a\hat u_s|_{k+1,l^{\hat K}_2}+|\hat a\hat u_s|_{k,l^{\hat K}_2})[\hat v]_{2,l^{\hat K}_2} = \mathcal O(h^{k+2})\|a\|_{k+1,\infty, l^e_2}\|u\|_{k+2,l^e_2}[v]_{2,l^{e}_2},
\end{align*}
and
\begin{align*}
\hat E((\hat a\hat u_s)_I\Pi_1 \hat v)  = 0.
\end{align*}

For the third term in \eqref{1dquadontint}, we sum them up over all the elements. Then for the line integral along $L_2$
\begin{align*}
&\sum_{e \cap L_2\neq\emptyset}\int_{-1}^1 (\hat a\hat u_s)_I(-1,t)\hat v(-1,t) dt  - \sum_{e \cap L_2\neq\emptyset} \int_{-1}^1 (\hat a\hat u_s)_I(-1,t)\hat v(-1,t) d^ht \\ 
= &\sum_{e \cap L_2\neq\emptyset}\hat E((\hat a\hat u_s)_I\hat v )
=  \sum_{e \cap L_2\neq\emptyset} \mathcal O(h^{k+2}) \|a\|_{k+1,\infty, l^e_2}\|u\|_{k+2,l^e_2}|v|_{2,l^{e}_2}.
\end{align*}

Let $s_\alpha$ and $\omega_\alpha$ ($\alpha=1,2,\cdots, k+2$) denote the quadrature points and weights in $(k+2)$-point Gauss-Lobatto quadrature rule for $s\in [-1,1]$.   Since $\hat v^2_{tt}(s, t)\in Q^{2k}(\hat K)$, $(k+2)$-point Gauss-Lobatto quadrature is exact for $s$-integration thus 
\[\int_{-1}^1\int_{-1}^1 \hat v_{tt}^2(s, t)dsdt=\sum_{\alpha=1}^{k+2}\omega_\alpha \int_{-1}^1 \hat v_{tt}^2(s_\alpha, t)dt,\]
which implies
\begin{equation}
\int_{-1}^1 \hat v_{tt}^2(\pm1, t)dt \leq C\int_{-1}^1\int_{-1}^1 \hat v_{tt}^2(s, t)dsdt,
\label{trace-polynomial}
\end{equation}
thus
\begin{align*}
h^{\frac{1}{2}}|v|_{2,l^{e}_2}\leq C [v]_{2,e}.
\end{align*} 
By Cauchy-Schwarz inequality and trace inequality, we have
\begin{align*}
& \sum_{e \cap L_2\neq\emptyset} \left (\left.\int_{-1}^1 (\hat a\hat u_s)_I\hat v dt \right|^{s=1}_{s=-1} - \left. \int_{-1}^1 (\hat a\hat u_s)_I\hat v d^ht \right|^{s=1}_{s=-1} \right) \\
 =& \sum_{e \cap L_2\neq\emptyset} \mathcal O(h^{k+2}) \|a\|_{k+1,\infty, l^e_2}\|u\|_{k+2,l^e_2}|v|_{2,l^{e}_2}\\
= & \sum_{e \cap L_2\neq\emptyset} \mathcal O(h^{k+\frac{3}{2}}) \|a\|_{k+1,\infty,l^e_2}\|u\|_{k+2,l^e_2}|v|_{2,e} = \mathcal O(h^{k+\frac{3}{2}}) \|a\|_{k+1,\infty,\Omega}\|u\|_{k+2,L_2}|v|_{2,\Omega} \\
= & \mathcal O(h^{k+\frac{3}{2}}) \|a\|_{k+1,\infty,\Omega}\|u\|_{k+3,\Omega}|v|_{2,\Omega}.
\end{align*}

Combine all the estimates above, we get \eqref{a-ah-auxvxv}. Since the $\frac 12$ order loss is only due to the line integral along the boundary $\partial \Omega$. If $v\in V_0^h$,  $v_{yy}=0$ on $L_2$ and $L_4$ so we have
 \eqref{a-ah-auxvxv0}.
\end{proof}

\section{Superconvergence of bilinear forms}
\label{sec-bilinear}
The M-type projection in \cite{MR635547, chen2001structure} is a very convenient tool for discussing the superconvergence of function values. Let $u_p$ be the M-type  $Q^k$ projection of the smooth exact solution $u$ and its definition will be given in the following subsection.
To establish the superconvergence of the original finite element method \eqref{intro-scheme1} for a generic elliptic problem \eqref{bilinearform} with smooth coefficients, 
one can show the following superconvergence of bilinear forms, see \cite{chen2001structure, lin1996} (see also \cite{li2019superconvergence} for a detailed proof):
\[
A(u-u_p,v_h)=
\begin{cases}
\mathcal O(h^{k+2})\|u\|_{k+3}\|v_h\|_2,\quad \forall v_h\in V^h_0,\\
\mathcal O(h^{k+\frac32})\|u\|_{k+3}\|v_h\|_2,\quad \forall v_h\in V^h. 
\end{cases}
\]
In this section we will show the superconvergence of the bilinear form $A_h$:
\begin{subnumcases}
{A_h(u-u_p,v_h)=}
\mathcal O(h^{k+2})\|u\|_{k+3}\|v_h\|_2,\quad \forall v_h\in V_0^h,\label{ellip-u-up-1}\\
\mathcal O(h^{k+\frac32})\|u\|_{k+3}\|v_h\|_2,\quad \forall v_h\in V^h\label{ellip-u-up-2}.
\end{subnumcases}

\subsection{Definition of M-type projection}
\label{sec-m-projection}
We first recall the definition of M-type projection. More detailed definition can also be found in  \cite{li2019superconvergence}. Legendre polynomials on the reference 
interval $[-1,1]$ are given as
\[l_k(t)=\frac{1}{2^k k!}\frac{d^k}{dt^k} (t^2-1)^k: l_0(t)=1, l_1(t)=t, l_2(t)=\frac12(3t^2-1), \cdots,\]
which are $L^2$-orthogonal to one another.
Define their antiderivatives as M-type polynomials:
\[M_{k+1}(t)=\frac{1}{2^k k!}\frac{d^{k-1}}{dt^{k-1}} (t^2-1)^k: M_0(t)=1, M_1(t)=t, M_2(t)=\frac12(t^2-1), M_3(t)=\frac12(t^3-t),\cdots.\]
which satisfy the following properties:
\begin{itemize}
 \item If $j-i\neq 0, \pm2$, then $M_i(t)\perp M_j(t)$, i.e., $\int_{-1}^1 M_i(t)M_j(t) dt=0.$
 \item Roots of $M_k(t)$ are the $k$-point Gauss-Lobatto quadrature points for $[-1,1]$. 
\end{itemize}
Since Legendre polynomials form a complete orthogonal basis for $L^{2}([-1,1])$,
for any $\hat f(t)\in H^1([-1,1])$,  its derivative $\hat f'(t)$ can be expressed as
Fourier-Legendre series
\[\hat f'(t)=\sum_{j=0}^{\infty}\hat b_{j+1}l_j(t), \quad \hat b_{j+1}=(j+\frac12)\int_{-1}^1 \hat f'(t)l_j(t)dt.\]
 The one-dimensional M-type projection is defined as
$
\hat f_k(t)=\sum_{j=0}^{k}\hat b_{j}M_j(t),
$ 
where $\hat b_0=\frac{\hat f(1)+\hat f(-1)}{2}$ is determined by $\hat b_1=\frac{\hat f(1)-\hat f(-1)}{2}$
so that $\hat f_k(\pm 1)=\hat f(\pm 1)$. 
We have
$
\hat f(t)=\lim\limits_{k\to \infty}\hat f_k(t)=\sum\limits_{j=0}^{\infty}\hat b_{j}M_j(t).
$
The remainder $\hat R[\hat f]_k(t)$ of one-dimensional M-type projection  is
\[\hat R[\hat f]_k(t)=\hat f(t)-\hat f_k(t)=\sum_{j=k+1}^{\infty}\hat b_{j}M_j(t).\]

For a function $\hat f(s,t)\in H^2(\hat K)$ on the reference cell  $\hat K=[-1,1]\times[-1,1]$, its two-dimensional M-type expansion is given as  
\[\hat f(s,t)=\sum_{i=0}^\infty\sum_{j=0}^\infty \hat b_{i,j} M_i(s)M_j(t),\]
where 
\begin{align*}
\hat b_{0,0}&=\frac14[\hat f(-1,-1)+\hat f(-1,1)+\hat f(1,-1)+\hat f(1,1)],\\
\hat b_{0,j}, \hat b_{1,j}&=\frac{2j-1}{4}\int_{-1}^1 [\hat f_t(1,t)\pm \hat f_t(-1,t)]l_{j-1}(t)dt, \quad j\geq 1,\\
\hat b_{i,0}, \hat b_{i,1}&=\frac{2i-1}{4}\int_{-1}^1 [\hat f_s(s,1)\pm \hat f_s(s,-1)]l_{i-1}(s)ds, \quad i\geq 1,\\
\hat b_{i,j}&=\frac{(2i-1)(2j-1)}{4}\iint_{\hat K}\hat f_{st}(s,t)l_{i-1}(s)l_{j-1}(t)dsdt, 
\quad i,j\geq 1.\end{align*}
The  M-type $Q^k$ projection of $\hat f$ on $\hat K$ and its remainder are defined as
\[\hat f_{k,k}(s,t)=\sum_{i=0}^k\sum_{j=0}^k \hat b_{i,j} M_i(s)M_j(t), \quad \hat R[\hat f]_{k,k}(s,t)=\hat f(s,t)-\hat f_{k,k}(s,t).\]
The  M-type $Q^k$ projection is equivalent to the point-line-plane interpolation used in \cite{lin1991rectangle, lin1996}. See \cite{li2019superconvergence} for the proof of the following fact:
\begin{theorem}
\label{plp-projection-theorem}
 The  M-type $Q^k$ projection is equivalent to the $Q^k$ point-line-plane projection $\Pi$ defined as
 follows:
 \begin{enumerate}
  \item $\Pi \hat u=\hat u$ at four corners of $\hat K=[-1,1]\times[-1,1]$.
  \item $\Pi \hat u-\hat u$ is orthogonal to polynomials of degree $k-2$ on each edge of $\hat K$.
  \item $\Pi \hat u-\hat u$ is orthogonal to any $\hat v\in Q^{k-2}(\hat K)$ on $\hat K$. 
  \end{enumerate}
\end{theorem}
For $f(x,y)$ on $e=[x_e-h, x_e+h]\times [y_e-h, y_e+h]$, let $\hat f(s,t)=
f( sh+x_e,  t h+y_e)$ then  the M-type $Q^k$ projection of $f$ on $e$ and its remainder are defined as 
\[f_{k,k}(x,y)=\hat f_{k,k}(\frac{x-x_e}{h},\frac{y-y_e}{h}), \quad R[ f]_{k,k}(x,y)= f(x,y)- f_{k,k}(x,y).\]
Now consider a function $u(x, y)\in H^{k+2} (\Omega)$, let
$u_p (x, y)$ denote its piecewise M-type $Q^k$ projection on each element $e$ in the mesh
$\Omega_h$. The first two properties in Theorem \ref{plp-projection-theorem} imply that $u_p (x, y)$  on each edge of $e$ is
uniquely determined by 
$u(x, y)$ along that edge. So $u_p (x, y)$  is a piecewise continuous $Q^k$ polynomial on $\Omega_h$.

M-type projection has the following properties. See \cite{li2019superconvergence} for the proof.

\begin{theorem}
\label{thm-superapproximation}
 \[\|u-u_p\|_{2,Z_0}=\mathcal O(h^{k+2}) \|u\|_{k+2},\quad\forall u\in H^{k+2}(\Omega).\]
 \[\|u-u_p\|_{\infty,Z_0}=\mathcal O(h^{k+2}) \|u\|_{k+2,\infty},
 \quad\forall u\in W^{k+2,\infty}(\Omega).\]
\end{theorem}

\begin{lemma}\label{estimateslemma}
 For $\hat f \in H^{k+1}(\hat K)$, $k \geq 2$,
 \begin{enumerate} 
  \item
  $|\hat R[\hat f]_{k,k}|_{0,\infty,\hat K}\leq C [\hat f]_{k+1,\hat K},\quad |\partial_s \hat R[\hat f]_{k,k}|_{0,\infty, \hat K}\leq C[\hat f]_{k+1,\hat K}.$
\item $\hat R[\hat f]_{k+1,k+1}-\hat R[\hat f]_{k,k}=M_{k+1}(t)\sum_{i=0}^k \hat b_{i,k+1}M_{i}(s)+M_{k+1}(s)\sum_{j=0}^{k+1}\hat b_{k+1,j}M_j(t).$ 
 \item 
 $|\hat b_{i,k+1}|\leq C_k |\hat f|_{k+1,2,\hat K},|\hat b_{k+1,i}|\leq C_k |\hat f|_{k+1,2,\hat K},\quad 0\leq i\leq k+1.$
  \item If $\hat f\in H^{k+2}(\hat K)$, then $|\hat b_{i,k+1}|\leq C_k |\hat f|_{k+2,2,\hat K},\quad 1\leq i\leq k+1.$
\end{enumerate}
\end{lemma}

\subsection{Estimates of M-type projection with quadrature}

\begin{lemma}\label{quad-int-rs}
Assume $\hat f(s,t) \in H^{k+3}(\hat K)$,
\[
\langle \hat R[\hat f]_{k+1,k+1}- \hat R[\hat f]_{k,k},  1 \rangle_{\hat K}=0,\quad |\langle \partial_s \hat R[\hat f]_{k+1,k+1} , 1 \rangle_{\hat K}| \leq C |\hat f|_{k+3,\hat K}.
\]
\end{lemma}
\begin{proof}
First, we have 
\begin{align*}
\langle \hat R[\hat f]_{k+1,k+1}- \hat R[\hat f]_{k,k},  1 \rangle_{\hat K}=\langle M_{k+1}(t)\sum_{i=0}^k \hat b_{i,k+1}M_{i}(s)+M_{k+1}(s)\sum_{j=0}^{k+1}\hat b_{k+1,j}M_j(t) , 1 \rangle_{\hat K}=0
\end{align*}
due to the fact that roots of $M_{k+1}(t)$ are the $(k+1)$-point Gauss-Lobatto quadrature points for $[-1,1]$.

We have
\begin{align*}
&\langle \partial_s\hat R[\hat f]_{k+1,k+1}  , 1 \rangle_{\hat K} \\
=& \langle \partial_s\hat R[\hat f]_{k+2,k+2} , 1 \rangle_{\hat K} - \langle \partial_s(\hat R[\hat f]_{k+2,k+2}-\hat R[\hat f]_{k+1,k+1}) , 1 \rangle_{\hat K}\\
= & \langle \partial_s\hat R[\hat f]_{k+2,k+2} , 1 \rangle_{\hat K} - \langle M_{k+2}(t)\sum_{i=0}^{k+1} \hat b_{i,k+2}M_{i}'(s)+M_{k+2}'(s)\sum_{j=0}^{k+2}\hat b_{k+2,j}M_j(t),1 \rangle_{\hat K}\\
= &  \langle \partial_s\hat R[\hat f]_{k+2,k+2} , 1 \rangle_{\hat K} - \langle M_{k+2}(t)\sum_{i=0}^k \hat b_{i+1,k+2}l_{i}(s),1 \rangle_{\hat K}+ \langle l_{k+1}(s)\sum_{j=0}^{k+2}\hat b_{k+2,j}M_j(t),1 \rangle_{\hat K}.
\end{align*}
Then by Lemma \ref{estimateslemma}, 
\[|\langle \partial_s\hat R[\hat f]_{k+2,k+2} , 1 \rangle_{\hat K}|\leq  C |\hat f|_{k+3,\hat K}.\]
Notice that we have $\langle l_{k+1}(s)\sum_{j=0}^{k+2}\hat b_{k+2,j}M_j(t),1 \rangle_{\hat K}=0$ since the $(k+1)$-point Gauss-Lobatto quadrature for $s$-integration is exact and $l_{k+1}(s)$ is orthogonal to $1$. Lemma \ref{estimateslemma} implies $|\hat b_{i+1,k+2}|\leq C[\hat f]_{k+3,\hat K}$ for $i\geq 0$, thus we have $$|\langle M_{k+2}(t)\sum_{i=0}^k \hat b_{i+1,k+2}l_{i}(s),1 \rangle_{\hat K}|\leq C[\hat f]_{k+3,\hat K}.$$ 
\end{proof}
\begin{lemma}
\label{lemma-bilinear-auxvx}
Assume $a(x,y)\in W^{k,\infty}(\Omega).$ Then
 \[
  \langle a (u-u_p)_x, (v_h)_x\rangle_h=\mathcal O(h^{k+2})\|a\|_{2,\infty}\|u\|_{k+3}\|v_h\|_2,\quad \forall v_h\in V^h.
\]
\end{lemma}
\begin{proof}
As before, we ignore the subscript of $v_h$ for simplicity.
We have
$$\langle a (u-u_p)_x, v_x\rangle_h=\sum_e \langle a (u-u_p)_x, v_x\rangle_{e,h},$$ 
and on each cell $e$,  
\begin{align}
& \langle a(u-u_p)_x,v_x \rangle_{e,h}
=\langle (R[u]_{k,k})_x, av_x\rangle_{e,h}
= \langle (\hat R[\hat u]_{k,k})_s, \hat a\hat v_s\rangle_{\hat K}\nonumber\\
=&\langle (\hat R[\hat u]_{k+1,k+1})_s, \hat a\hat v_s\rangle_{\hat K}+\langle(\hat R[\hat u]_{k,k}-\hat R[\hat u]_{k+1,k+1})_s, \hat a\hat v_s\rangle_{\hat K}.\label{auxvx-remainder-highlow}
\end{align}
For the first term in \eqref{auxvx-remainder-highlow}, we have
\begin{align*}
&\langle (\hat R[\hat u]_{k+1,k+1})_s , \hat a \hat v_s\rangle_{\hat K}
=\langle (\hat R[\hat u]_{k+1,k+1})_s, \hat a \overline{\hat v_s}\rangle_{\hat K}+\langle (\hat R[\hat u]_{k+1,k+1})_s, \hat a( \hat v_s-\overline{\hat v_s})\rangle_{\hat K}.
\end{align*}
By Lemma \ref{quad-int-rs},
\begin{equation*}
\langle (\hat R[\hat u]_{k+1,k+1})_s, \overline{\hat a}\, \overline{\hat v_s}\rangle_{\hat K}\leq C|\hat a|_{0,\infty}|\hat u|_{k+3,\hat K}|\hat v|_{1,\hat{K}}.
\end{equation*}
By Lemma \ref{estimateslemma},
\[|(\hat R[\hat u]_{k+1,k+1})_s|_{0,\infty,\hat K}\leq C[\hat u]_{k+2,\hat K}.\] By Bramble-Hilbert Lemma Theorem \ref{bh-lemma} we have
\begin{align*}
& \langle (\hat R[\hat u]_{k+1,k+1})_s, \hat a \overline{\hat v_s}\rangle_{\hat K} =\langle (\hat R[\hat u]_{k+1,k+1})_s, \overline{\hat a}\, \overline{\hat v_s}\rangle_{\hat K} +\langle (\hat R[\hat u]_{k+1,k+1})_s, (\hat a-\overline{\hat a})\overline{\hat v_s}\rangle_{\hat K}\\
\leq & C(|\hat a|_{0,\infty}|\hat u|_{k+3,\hat K}|\hat v|_{1,\hat{K}}+|\hat a-\overline{\hat a}|_{0,\infty}|\hat u|_{k+2,\hat K}|\hat v|_{1,\hat K})\\
\leq & C(|\hat a|_{0,\infty}|\hat u|_{k+3,\hat K}|\hat v|_{1,\hat{K}}+|\hat a|_{1,\infty}|\hat u|_{k+2,\hat K}|\hat v|_{1,\hat K})
=\mathcal O(h^{k+2}) \|a\|_{1,\infty,e}\|u\|_{k+3,e}\|v\|_{1,e},
\end{align*}
and
\begin{align*}
& \langle (\hat R[\hat u]_{k+1,k+1})_s , \hat a( \hat v_s -\overline{\hat v_s})\rangle_{\hat K}
\leq  C[\hat u]_{k+2,2,\hat K}|\hat a|_{0,\infty,\hat K}|\hat v_s -\overline{\hat v_s}|_{0,\infty,\hat K}\\
\leq & C[\hat u]_{k+2,2,\hat K}|\hat a|_{0,\infty,\hat K}|\hat v_s -\overline{\hat v_s}|_{0,2,\hat K} = \mathcal O(h^{k+2})[ u]_{k+2,2,e}|a|_{0,\infty,e}|v|_{2,2,e}.
\end{align*}
Thus, 
\begin{equation}\label{auxvx-1stterm}
\langle (\hat R[\hat u]_{k+1,k+1})_s , \hat a \hat v_s \rangle_{\hat K}  = \mathcal O(h^{k+2}) \|a\|_{1,\infty,e}| u|_{k+3,2,e} \|v\|_{2,e}.
\end{equation}

For the second term in \eqref{auxvx-remainder-highlow}, we have
\begin{align}
&\langle(\hat R[\hat u]_{k,k}-\hat R[\hat u]_{k+1,k+1})_s, \hat a\hat v_s\rangle_{\hat K}\nonumber\\
=&-\langle(M_{k+1}(t)\sum_{i=0}^k \hat b_{i,k+1}M_{i}(s)+M_{k+1}(s)\sum_{j=0}^{k+1}\hat b_{k+1,j}M_j(t))_s, \hat a\hat v_s\rangle_{\hat K}\nonumber\\
=&-\langle M_{k+1}(t)\sum_{i=0}^{k-1} \hat b_{i+1,k+1}l_{i}(s)+l_{k}(s)\sum_{j=0}^{k+1}\hat b_{k+1,j}M_j(t), \hat a\hat v_s\rangle_{\hat K}\nonumber\\
=&-\langle M_{k+1}(t)\sum_{i=0}^{k-1} \hat b_{i+1,k+1}l_{i}(s), \hat a\hat v_s\rangle_{\hat K} - \langle l_{k}(s)\sum_{j=0}^{k+1}\hat b_{k+1,j}M_j(t), \hat a\hat v_s\rangle_{\hat K}.\label{R2-R3}
\end{align}
Since $M_{k+1}(t)$ vanishes at $(k+1)$ Gauss-Lobatto points, we have  
 $$\langle M_{k+1}(t)\sum_{i=0}^{k-1} \hat b_{i+1,3}l_{i}(s), \hat a\hat v_s\rangle_{\hat K}=0.$$
 For the second term in \eqref{R2-R3}, 
 \begin{align*}
&\langle l_{k}(s)\sum_{j=0}^{k+1}\hat b_{k+1,j}M_j(t), \hat a\hat v_s\rangle_{\hat K}= \langle l_{k}(s)\sum_{j=0}^{k+1}\hat b_{k+1,j}M_j(t), \hat a\overline{\hat v_s}\rangle_{\hat K}+\langle l_{k}(s)\sum_{j=0}^{k+1}\hat b_{k+1,j}M_j(t), \hat a(\hat v_s- \overline{\hat v_s})\rangle_{\hat K}\\
= &\langle l_{k}(s)\sum_{j=0}^{k+1}\hat b_{k+1,j}M_j(t), (\hat a-\hat \Pi_1\hat a)\overline{\hat v_s}\rangle_{\hat K}+\langle l_{k}(s)\sum_{j=0}^{k+1}\hat b_{k+1,j}M_j(t), (\hat \Pi_1\hat a)\overline{\hat v_s}\rangle_{\hat K}\\
&+\langle l_{k}(s)\sum_{j=0}^{k+1}\hat b_{k+1,j}M_j(t),(\hat a-\overline{\hat a})(\hat v_s- \overline{\hat v_s})\rangle_{\hat K}+\langle l_{k}(s)\sum_{j=0}^{k+1}\hat b_{k+1,j}M_j(t),\overline{\hat a}(\hat v_s- \overline{\hat v}_s)\rangle_{\hat K}\\
= &\langle l_{k}(s)\sum_{j=0}^{k+1}\hat b_{k+1,j}M_j(t), (\hat a-\hat \Pi_1\hat a)\overline{\hat v}_s\rangle_{\hat K}+\langle l_{k}(s)\sum_{j=0}^{k+1}\hat b_{k+1,j}M_j(t),(\hat a-\overline{\hat a})(\hat v_s- \overline{\hat v}_s)\rangle_{\hat K},
\end{align*}
 where the last step  is due to the facts that $(\hat\Pi_1 \hat a) \overline{\hat v_s}$ and $\overline{\hat a}(\hat v_s- \overline{\hat v}_s)$ are polynomials of degree at most $k-1$ with respect to variable $s$, the $(k+1)$-point Gauss-Lobatto quadrature on $s$-integration is exact for polynomial of degree $2k-1$, and $l_k(s)$ is orthogonal to polynomials of lower degree. 
With Lemma \ref{estimateslemma}, we have 
\begin{align}\label{auxvx-2ndterm}
\langle l_{k}(s)\sum_{j=0}^{k+1}\hat b_{k+1,j}M_j(t), \hat a\hat v_s\rangle_{\hat K}
\leq C|\hat u|_{k+1,2,\hat K}(|\hat a|_{2,\infty}|\hat v|_{1,\hat K}+ |\hat a|_{1,\infty}|\hat v|_{2,\hat K})=\mathcal O(h^{k+2}) \|a\|_{2,\infty}\|u\|_{k+1, e}\|v\|_{2,e}.
\end{align}
Combined with \eqref{auxvx-1stterm}, we have proved the  estimate.
\end{proof}

\begin{lemma}\label{lemma-bilinear-cuv}
Assume $a(x,y)\in W^{2,\infty}(\Omega).$ Then
 \[
  \langle a (u-u_p), v_h\rangle_h=\mathcal O(h^{k+2})\|a\|_{2,\infty}\|u\|_{k+2}\|v_h\|_2,\quad \forall v_h\in V^h.
\]
\end{lemma}
\begin{proof}
As before, we ignore the subscript of $v_h$ for simplicity and 
$$\langle a(u-u_p), v\rangle_h=\sum_e \langle a(u-u_p), v\rangle_{e,h}.$$ 
On each cell $e$ we have 
\begin{align}\label{auv}
 \langle a(u-u_p), v\rangle_{e,h}
=\langle R[u]_{k,k}, av\rangle_{e,h}
= h^2\langle \hat R[\hat u]_{k,k}, \hat a\hat v\rangle_{\hat K}=h^2\langle \hat R[\hat u]_{k,k}, \hat a\hat v- \overline{\hat a\hat v} \rangle_{\hat K}+h^2\langle\hat R[\hat u]_{k,k},  \overline{\hat a\hat v} \rangle_{\hat K}.
\end{align}
For the first term in \eqref{auv}, due to the embedding $H^{2}(\hat K)\hookrightarrow C^0(\hat K)$, Bramble-Hilbert Lemma Theorem \ref{bh-lemma} and Lemma \ref{estimateslemma}, we have
 \begin{align*}
&h^2\langle \hat R[\hat u]_{k,k}, \hat a\hat v- \overline{\hat a\hat v} \rangle_{\hat K}
\leq C h^2 |R[\hat u]_{k,k}|_{\infty}|\hat a\hat v- \overline{\hat a\hat v}|_{\infty} \leq C h^2|\hat{u}|_{k+1,\hat K}\|\hat a\hat v- \overline{\hat a\hat v}\|_{2,\hat K}\\
&\leq C h^2|\hat{u}|_{k+1,\hat K}(\|\hat a\hat v- \overline{\hat a\hat v}\|_{L^2(\hat K)}+|\hat a\hat v|_{1,\hat K}+|\hat a\hat v|_{2,\hat K})\\
&\leq C h^2|\hat{u}|_{k+1,\hat K}(|\hat a\hat v|_{1,\hat K}+|\hat a\hat v|_{2,\hat K}) =\mathcal O(h^{k+2})\|a\|_{2,\infty, e}\|u\|_{k+1,e}\|v\|_{2,e}.
\end{align*}
For the second term in \eqref{auv}, we have 
\begin{align*}
h^2\langle\hat R[\hat u]_{k+1,k+1},  \overline{\hat a\hat v} \rangle_{\hat K}=h^2\langle\hat R[\hat u]_{k+1,k+1},  \overline{\hat a\hat v} \rangle_{\hat K}-h^2\langle \hat R[\hat u]_{k+1,k+1}- \hat R[\hat u]_{k,k},  \overline{\hat a\hat v} \rangle_{\hat K}.
\end{align*}
By Lemma \ref{estimateslemma} and Lemma \ref{quad-int-rs}  we have
\[h^2\langle\hat R[\hat u]_{k+1,k+1},  \overline{\hat a\hat v} \rangle_{\hat K}\leq Ch^2  |\hat u|_{k+2,\hat K}|\hat a \hat v|_{0,\hat K}=\mathcal O(h^{k+2})\|a\|_{0,\infty, e}\| u\|_{k+2,e}\| v\|_{0,e},\]
and \[h^2\langle \hat R[\hat u]_{k+1,k+1}- \hat R[\hat u]_{k,k},  \overline{\hat a\hat v} \rangle_{\hat K}=0.\]
Thus, we have $\langle a (u-u_p), v_h\rangle_h=\mathcal O(h^{k+2})\|a\|_{2,\infty}\|u\|_{k+2}\|v_h\|_2.$
\end{proof}

\begin{lemma}\label{lemma-bilinear-buxv}
Assume $a(x,y)\in W^{2,\infty}(\Omega).$ Then
 \[
  \langle a (u-u_p)_x, v_h\rangle_h=\mathcal O(h^{k+2})\|a\|_{2,\infty}\|u\|_{k+3}\|v_h\|_2,\quad \forall v_h\in V^h.
\]
\end{lemma}
\begin{proof}
As before, we ignore the subscript in $v_h$ and we have
$$\langle a(u-u_p)_x, v\rangle_h=\sum_e \langle a(u-u_p)_x, v\rangle_{e,h}.$$ 
On each cell $e$, we have 
\begin{align}
& \langle a(u-u_p)_x, v\rangle_{e,h}
=\langle (R[u]_{k,k})_x, av\rangle_{e,h}
= h\langle (\hat R[\hat u]_{k,k})_s, \hat a\hat v\rangle_{\hat K}\nonumber\\
=&h\langle (\hat R[\hat u]_{k+1,k+1})_s, \hat a\hat v\rangle_{\hat K}-h\langle(\hat R[\hat u]_{k+1,k+1}-\hat R[\hat u]_{k,k})_s, \hat a\hat v\rangle_{\hat K}\label{buxv}.
\end{align}
For the first term in \eqref{buxv}, we have
\begin{align*}
&\langle (\hat R[\hat u]_{k+1,k+1})_s , \hat a \hat v\rangle_{\hat K}
\leq \langle (\hat R[\hat u]_{k+1,k+1})_s, \overline{\hat a \hat v}\rangle_{\hat K}+\langle (\hat R[\hat u]_{k+1,k+1})_s, \hat a \hat v-\overline{\hat a \hat v}\rangle_{\hat K}
\end{align*}
Due to Lemma \ref{quad-int-rs},
\[
h\langle (\hat R[\hat u]_{k+1,k+1})_s, \overline{\hat a \hat v}\rangle_{\hat K} \leq Ch\|a\|_{0,\infty}|u|_{k+3,\hat K}\|v\|_{0,\hat K}=\mathcal O(h^{k+2})\|a\|_{0,\infty} \|u\|_{k+3,e}\|v\|_{0,e},
\]
and by the same arguments as in the proof of Lemma \ref{lemma-bilinear-cuv} we have
\begin{align*}
 & h\langle (\hat R[\hat u]_{k+1,k+1})_s, \hat a \hat v-\overline{\hat a \hat v}\rangle_{\hat K}
\leq C h |(R[\hat u]_{k+1,k+1})_s|_{\infty}|\hat a\hat v- \overline{\hat a\hat v}|_{\infty}\leq C h|\hat{u}|_{k+2,\hat K}\|\hat a\hat v- \overline{\hat a\hat v}\|_{2,\hat K}\\
\leq & C h|\hat{u}|_{k+2,\hat K}(\|\hat a\hat v- \overline{\hat a\hat v}\|_{L^2(\hat K)}+|\hat a\hat v|_{1,\hat K}+|\hat a\hat v|_{2,\hat K})
\leq  C h|\hat{u}|_{k+2,\hat K}(|\hat a\hat v|_{1,\hat K}+|\hat a\hat v|_{2,\hat K}) =\mathcal  O(h^{k+2})\|a\|_{2,\infty} \|u\|_{k+2,e}\|v\|_{2,e}.
\end{align*}
Thus 
\begin{equation}\label{buxv-1stterm}
h\langle (\hat R[\hat u]_{k+1,k+1})_s, \hat a\hat v\rangle_{\hat K}=\mathcal O(h^{k+2})\|a\|_{2,\infty} \|u\|_{k+3,e}\|v\|_{2,e}.
\end{equation}

For the second term in \eqref{buxv}, we have 
\begin{align*}
&\langle(\hat R[\hat u]_{k+1,k+1}-\hat R[\hat u]_{k,k})_s, \hat a\hat v\rangle_{\hat K}\nonumber\\
=&\langle(M_{k+1}(t)\sum_{i=0}^k \hat b_{i,k+1}M_{i}(s)+M_{k+1}(s)\sum_{j=0}^{k+1}\hat b_{k+1,j}M_j(t))_s, \hat a\hat v\rangle_{\hat K}\nonumber\\
=&\langle M_{k+1}(t)\sum_{i=0}^{k-1} \hat b_{i+1,k+1}l_{i}(s)+l_{k}(s)\sum_{j=0}^{k+1}\hat b_{k+1,j}M_j(t), \hat a\hat v\rangle_{\hat K}\nonumber\\
=&\langle M_{k+1}(t)\sum_{i=0}^{k-1} \hat b_{i+1,k+1}l_{i}(s), \hat a\hat v\rangle_{\hat K} + \langle l_{k}(s)\sum_{j=0}^{k+1}\hat b_{k+1,j}M_j(t), \hat a\hat v\rangle_{\hat K}\\
= &\langle l_{k}(s)\sum_{j=0}^{k+1}\hat b_{k+1,j}M_j(t),  \hat a\hat v\rangle_{\hat K},
\end{align*}
where the last step is due to that $M_{k+1}(t)$ vanishes at $(k+1)$ Gauss-Lobatto points.
 Then
 \begin{align*}
&\langle(\hat R[\hat u]_{k,k}-\hat R[\hat u]_{k+1,k+1})_s, \hat a\hat v\rangle_{\hat K}=\langle l_{k}(s)\sum_{j=0}^{k+1}\hat b_{k+1,j}M_j(t), \hat a\hat v\rangle_{\hat K}\\
= &\langle l_{k}(s)\sum_{j=0}^{k+1}\hat b_{k+1,j}M_j(t), \hat a\hat v-\hat\Pi_1( \hat a\hat v)\rangle_{\hat K}+\langle l_{k}(s)\sum_{j=0}^{k+1}\hat b_{k+1,j}M_j(t), \hat\Pi_1( \hat a\hat v) \rangle_{\hat K}\\
= &\langle l_{k}(s)\sum_{j=0}^{k+1}\hat b_{k+1,j}M_j(t), \hat a\hat v-\hat\Pi_1( \hat a\hat v)\rangle_{\hat K},
\end{align*}
 where the last step is due to the facts that $\hat\Pi_1( \hat a\hat v)$ is a linear function in $s$  thus the $(k+1)$-point Gauss-Lobatto quadrature on $s$-variable is exact, and $l_k(s)$ is orthogonal to linear functions. 

By Lemma \ref{estimateslemma} and Theorem \ref{bh-lemma}, we have 
\begin{align*}
&\langle(\hat R[\hat u]_{k,k}-\hat R[\hat u]_{k+1,k+1})_s, \hat a\hat v\rangle_{\hat K}=\langle l_{k}(s)\sum_{j=0}^{k+1}\hat b_{k+1,j}M_j(t), \hat a\hat v-\hat\Pi_1( \hat a\hat v)\rangle_{\hat K}\\
\leq & C|u|_{k+1,\hat K}|\hat a\hat v|_{2,\hat K}\leq C|u|_{k+1,\hat K}(|\hat a|_{2,\infty,\hat K}|\hat v|_{0,\hat K}+|\hat a|_{1,\infty,\hat K}|\hat v|_{1,\hat K}+|\hat a|_{0,\infty}|\hat v|_{2,\hat K})
\end{align*}
Thus 
\begin{equation}\label{buxv-2ndterm}
h\langle(\hat R[\hat u]_{k,k}-\hat R[\hat u]_{k+1,k+1})_s, \hat a\hat v\rangle_{\hat K}=\mathcal O(h^{k+2})\|a\|_{2,\infty}\|u\|_{k+1,e}\|v\|_{2,e}.
\end{equation}
  By \eqref{buxv-1stterm} and \eqref{buxv-2ndterm} and sum up over all the cells, we get the desired estimate.
\end{proof}

\begin{lemma}
\label{lemma-bilinear-auxvy}
Assume $a(x,y)\in W^{4,\infty}(\Omega).$ Then
\begin{subnumcases}
{\langle a (u-u_p)_x, (v_h)_y\rangle_h=}
\mathcal O(h^{k+\frac32})\|a\|_{k+2,\infty}\|u\|_{k+3}\|v_h\|_2,\quad \forall v_h\in V^h,\label{crossterm-1}\\
\mathcal O(h^{k+2})\|a\|_{k+2,\infty}\|u\|_{k+3}\|v_h\|_2,\quad \forall v_h\in V_0^h.\label{crossterm-2}
\end{subnumcases}
\end{lemma}
\begin{proof}
 We ignore the subscript in $v_h$ and we have
 $$\langle a (u-u_p)_x, v_y\rangle_h=\sum_e \langle a (u-u_p)_x, v_y\rangle_{e,h},$$ 
and on each cell $e$  
\begin{align}
& \langle a(u-u_p)_x,v_y \rangle_{e,h}
=\langle (R[u]_{k,k})_x, av_y\rangle_{e,h}
= \langle (\hat R[\hat u]_{k,k})_s, \hat a\hat v_t\rangle_{\hat K}\nonumber\\
=&\langle (\hat R[\hat u]_{k+1,k+1})_s, \hat a\hat v_t\rangle_{\hat K}+\langle(\hat R[\hat u]_{k,k}-\hat R[\hat u]_{k+1,k+1})_s, \hat a\hat v_t\rangle_{\hat K}.\label{auxvy-remainder-highlow}
\end{align}
By the same arguments as in the proof of Lemma \ref{lemma-bilinear-auxvx}, we have
\begin{equation}\label{auxvy-1stterm}
\langle (\hat R[\hat u]_{k+1,k+1})_s , \hat a \hat v_t \rangle_{\hat K}  = \mathcal O(h^{k+2}) \|a\|_{1,\infty} |u|_{k+3,2,e} \|v\|_{2,e},
\end{equation}
and
\begin{align*}
\langle(\hat R[\hat u]_{k,k}-\hat R[\hat u]_{k+1,k+1})_s, \hat a\hat v_t\rangle_{\hat K}
= - \langle l_{k}(s)\sum_{j=0}^{k+1}\hat b_{k+1,j}M_j(t), \hat a\hat v_t\rangle_{\hat K}.
\end{align*}
For simplicity, we define
$$\hat b_{k+1}(t):=\sum_{j=0}^{k+1}\hat b_{k+1,j}M_j(t).$$
then by the third and fourth estimates in Lemma \ref{estimateslemma}, we have 
\begin{align*}
&|\hat b_{k+1}(t)|\leq C\sum_{j=0}^{k+1} |\hat b_{k+1,j}| \leq C |\hat u|_{k+1,\hat K},\\
 &|\hat b_{k+1}^{(m)}(t)|\leq C\sum_{j=m}^{k+1} |\hat b_{k+1,j}| \leq C |\hat u|_{k+2,\hat K}, \quad 1\leq m .
 \end{align*}
We use the same technique in the proof of Theorem \ref{a-ah},
\begin{align*}
&\langle(\hat R[\hat u]_{k,k}- \hat R[\hat u]_{k+1,k+1})_s, \hat a\hat v_t\rangle_{\hat K}
= -\langle l_{k}(s)\hat b_{k+1}(t), \hat a\hat v_t\rangle_{\hat K}\\
= & -\iint_{\hat K}l_{k}(s)\hat b_{k+1}(t)\hat a \hat v_td^hsd^ht
=- \iint_{\hat K}(l_{k}\hat b_{k+1}\hat a)_I\hat v_td^hsd^ht\\
=&-\iint_{\hat K}(l_{k}\hat b_{k+1}\hat a)_I\hat v_td^hsd^ht +\iint_{\hat K}l_{k}\hat b_{k+1}\hat a\hat v_tdsdt - \iint_{\hat K}l_{k}\hat b_{k+1}\hat a\hat v_tdsdt.
\end{align*}
and 
\begin{align*}
&- \iint_{\hat K}(l_{k}\hat b_{k+1}\hat a)_I\hat v_td^hsd^ht + \iint_{\hat K}l_{k}\hat b_{k+1}\hat a\hat v_tdsdt  \\
= & \iint_{\hat K}\left[l_{k}\hat b_{k+1}\hat a-(l_{k}\hat b_{k+1}\hat a)_I\right]\hat v_tdsdt + \iint_{\hat K}(l_{k}\hat b_{k+1}\hat a)_I\hat v_tdsdt - \iint_{\hat K}(l_{k}\hat b_{k+1}\hat a)_I\hat v_td^hsdt\\
= & \iint_{\hat K}\left[l_{k}\hat b_{k+1}\hat a-(l_{k}\hat b_{k+1}\hat a)_I\right]\hat v_tdsdt + \iint_{\hat K}\partial_t(l_{k}\hat b_{k+1}\hat a)_I\hat vd^hsdt - \iint_{\hat K}\partial_t(l_{k}\hat b_{k+1}\hat a)_I\hat vdsdt\\
&+\left( \left.\int_{-1}^1 (l_{k}\hat b_{k+1}\hat a)_I\hat v ds \right|^{t=1}_{t=-1} - \left. \int_{-1}^1 (l_{k}\hat b_{k+1}\hat a)_I\hat v d^hs \right|^{t=1}_{t=-1} \right) = I + II + III.
\end{align*}

After integration by parts with respect to the variable $s$, we have
\begin{align*}
& \iint_{\hat K}l_{2}(s)\hat b_3(t)\hat a\hat v_t dsdt= -\iint_{\hat K}M_{3}(s)\hat b_3(t)(\hat a_s\hat v_t + \hat a\hat v_{st}) ds dt,
\end{align*}
which is exactly the same integral estimated in the proof of Lemma 3.7 in \cite{li2019superconvergence}. 
By the same proof of Lemma 3.7 in \cite{li2019superconvergence}, after summing over all elements, we have
the estimate for the term $\iint_{\hat K}l_{k}(s)\hat b_{k+1}(t)\hat a\hat v_t dsdt$:
\[\sum_e \iint_{\hat K}l_{k}(s)\hat b_{k+1}(t)\hat a\hat v_t dsdt= 
\begin{cases}
\mathcal O(h^{k+\frac32})\|a\|_{k+2,\infty}\|u\|_{k+3}\|v\|_2,& \forall v\in V^h,\\
\mathcal O(h^{k+2})\|a\|_{k+2,\infty}\|u\|_{k+3}\|v\|_2,& \forall v\in V_0^h. 
\end{cases}\]

Then we can do similar estimation as in Theorem \ref{a-ah} for $I,II,III$ separately.

For term $I$, by Theorem \ref{bh-lemma} and the estimate \eqref{quaderror-theorem-ref}, we have
\begin{align*}
&\iint_{\hat K}\left[l_{k}\hat b_{k+1}\hat a-(l_{k}\hat b_{k+1}\hat a)_I\right]\hat v_tdsdt\\
=&\iint_{\hat K}\left[l_{k}\hat b_{k+1}\hat a-(l_{k}\hat b_{k+1}\hat a)_I\right]\overline{\hat v_t}dsdt + \iint_{\hat K}\left[l_{k}\hat b_{k+1}\hat a-(l_{k}\hat b_{k+1}\hat a)_I\right](\hat v_t -\overline{\hat v_t})dsdt\\
\leq & C\left[l_{k}\hat b_{k+1}\hat a\right]_{k+2,\hat K}|\hat v|_{1,\hat K} + C\left[l_{k}\hat b_{k+1}\hat a\right]_{k+1,\hat K}|\hat v|_{2,\hat K}\\
\leq & C \left(\sum_{m=2}^{k+2}|\hat a|_{m,\infty,\hat K} \max_{t\in[-1,1]}|\hat b_{k+1}(t)| \right)|\hat v|_{1,\hat K} + C \left(\sum_{m=0}^{k+2}|\hat a|_{m,\infty,\hat K} \max_{t\in[-1,1]}|\hat b_{k+1}^{(k+2-m)}(t)| \right)|\hat v|_{1,\hat K} \\
+& C\left(\sum_{m=1}^{k+1}|\hat a|_{m,\infty,\hat K} \max_{t\in[-1,1]}|\hat b_{k+1}(t)|\right)|\hat v|_{2,\hat K} +  C\left(\sum_{m=0}^{k+1}|\hat a|_{m,\infty,\hat K} \max_{t\in[-1,1]}|\hat b_{k+1}^{(k+1-m)}(t)|\right)|\hat v|_{2,\hat K}\\
= &\mathcal O(h^{k+2})\|a\|_{k+2,\infty}\|u\|_{k+2,e}\|v\|_{2,e}.
\end{align*}
For term $II$, as in the proof of Theorem \ref{a-ah}, we define the linear form as
$$\hat E_{\hat v}(\hat f) = \iint_{\hat K}(\hat F_I)_t \hat v dsdt -\iint_{\hat K} (\hat F_I)_t\hat v d^hsdt,$$
for each $\hat v \in Q^k(\hat K)$ and   $\hat F$ is an antiderivative of $\hat f$ w.r.t. variable $t$. We can easily see that $\hat E_{\hat v}$ is well defined and $\hat E_{\hat v}$ is a continuous linear form on $H^k(\hat K)$.
With projection $\hat \Pi_1$ defined in \eqref{projection1}, we have 
$$
\hat E_{\hat v}(\hat f) = \hat E_{\hat v - \hat \Pi_1\hat v} (\hat f)+ \hat E_{\hat \Pi_1\hat v}(\hat f),\quad  \forall \hat v \in Q^k(\hat K).
$$ 
Since $ Q^{k-1}(\hat K) \subset \ker \hat E_{\hat v - \hat\Pi_1\hat v}$ thus 
$$
\hat E_{\hat v - \hat \Pi_1\hat v} (\hat f)  \leq C[f]_{k,\hat K}\|\hat v- \hat \Pi_1\hat v\|_{0,\hat K}\leq C[f]_{k,\hat K}|\hat v|_{2,\hat K}  
$$
 and 
 $$
\hat E_{\hat \Pi_1\hat v} (\hat f) =\iint_{\hat K}(\hat F_I)_t \hat \Pi_1\hat v dsdt -\iint_{\hat K} (\hat F_I)_t\hat \Pi_1\hat v d^hsdt=0.
$$
Thus we have 
 \begin{align*}
& \iint_{\hat K} \partial_t(l_k\hat b_{k+1}\hat a)_I\hat v d^hsdt -\iint_{\hat K}\partial_t(l_k\hat b_{k+1}\hat a)_I\hat v dsdt  
= -\hat E_{\hat v}((l_k\hat b_{k+1}\hat a)_t)  \\
=& -\hat E_{\hat v - \Pi_1\hat v} ((l_k\hat b_{k+1}\hat a)_t)
\leq  C [(l_k\hat b_{k+1}\hat a)_t]_{k,\hat K}|\hat v_h|_{2,\hat K} 
= \mathcal O(h^{k+2})\|a\|_{k+1,\infty, e}\|u\|_{k+2,e}|v|_{2,e}.
 \end{align*}

Now we only need to discuss term $III$. Let $L_1$ and $L_3$ denote the top and bottom boundaries of $\Omega$ and let $l^e_1$, $l^e_3$ denote the top and bottom edges of element $e$ (and $l^{\hat K}_1$ and $l^{\hat K}_3$ for $\hat K$). Notice that after mapping back to the cell $e$ we have
\begin{align*}
b_{k+1}(y_e+h)=\hat b_{k+1}(1)=\sum_{j=0}^{k+1} \hat b_{k+1,j} M_j(1)=\hat b_{k+1,0}+ \hat b_{k+1,1}\\
=(k+\frac12)\int_{-1}^1 \partial_s \hat u(s,1) l_k(s)ds 
=(k+\frac12)\int_{x_e-h}^{x_e+h} \partial_x  u(x,y_e+h) l_k(\frac{x-x_e}{h})dx,
\end{align*}
and similarly we get  $b_{k+1}(y_e-h)=\hat b_{k+1}(-1)=(k+\frac12)\int_{x_e-h}^{x_e+h} \partial_x  u(x,y_e-h) l_k(\frac{x-x_e}{h})dx$. Thus the term $l(\frac{x-x_e}{h})b_{k+1}(y)M_{k+1}av$ is continuous across the top and bottom edges of cells. Therefore, if summing over all 
elements $e$, the line integral on the inner edges are cancelled out.  So after summing over all elements, the line integral  reduces to two line integrals along $L_1$ and $L_3$. We only need to discuss one of them. 
For a cell $e$ adjacent to $L_1$,  consider its reference cell $\hat K$ and define linear form  $\hat E(\hat f ) = \int_{-1}^1 \hat f(s,1) ds - \int_{-1}^1 \hat f(s,1)d^hs$, then we have  
$$\hat E(\hat f \hat v)  \leq C |\hat f|_{0,\infty, l^{\hat K}_1} |\hat v|_{0,\infty,l^{\hat K}_1} \leq C\|\hat f\|_{2,l^{\hat K}_1}\|\hat v\|_{0,l^{\hat K}_1},$$ thus the mapping $\hat f \rightarrow \hat E(\hat f \hat v)$ is continuous with operator norm less than $C\|\hat v\|_{0,l^{\hat K}_1}$ for some $C$. Since $\hat E((\hat a\hat u_s)_I\hat \Pi_1 \hat v)=0$ we have 
\begin{align*}
 & \sum_{e\cap L_1\neq \emptyset}\int_{-1}^1 (l_{k}\hat b_{k+1}\hat a)_I\hat v ds- \int_{-1}^1 (l_{k}\hat b_{k+1}\hat a)_I\hat v d^hs\\
= & \sum_{e\cap L_1\neq \emptyset}\hat E((l_k\hat b_{k+1}\hat a)_I \hat v) = \sum_{e\cap L_1\neq \emptyset}\hat E((l_k\hat b_{k+1}\hat a)_I(\hat v-\hat \Pi_1 \hat v)) \leq \sum_{e\cap L_1\neq \emptyset}C[(l_k\hat b_{k+1}\hat a)_I]_{k,l^{\hat K}_1}[\hat v]_{2,l^{\hat K}_1} \\
 \leq & \sum_{e\cap L_1\neq \emptyset} C(|l_k\hat b_{k+1}\hat a-(l_k\hat b_{k+1}\hat a)_I|_{k,l^{\hat K}_1}+|l_k\hat b_{k+1}\hat a|_{k,l^{\hat K}_1})[\hat v]_{2,l^{\hat K}_1}\\
 \leq & \sum_{e\cap L_1\neq \emptyset} (|l_k\hat b_{k+1}\hat a|_{k+1,l^{\hat K}_1}+|l_k\hat b_{k+1}\hat a|_{k,l^{\hat K}_1})[\hat v]_{2,l^{\hat K}_1}
 \leq \sum_{e\cap L_1\neq \emptyset}C \|\hat a\|_{k,\infty,\hat K} |\hat b_{k+1}(1)|[\hat v]_{2,l^{\hat K}_1}.
\end{align*}
Since $l_k(t)=\frac{1}{2^k k!}\frac{d^k}{dt^k} (t^2-1)^k$, after integration by parts $k$ times, $$\hat b_{k+1}(1)=(k+\frac12)\int_{-1}^{1} \partial_s  u(s,1) l_k(s)dx =(-1)^k (k+\frac12)\int_{-1}^{1} \partial^{k+1}_s  u(s,1) L(s)ds, $$
where $L(s)$ is a polynomial of degree $2k$ by taking antiderivatives of $l_k(s)$ $k$ times.
Then by Cauchy-Schwarz inequality we have
\begin{align*}
\hat b_{k+1}(1) \leq C \left(\int_{-1}^1|\partial^{k+1}_s  \hat u(s,1)|^2ds\right)^{\frac12} \leq Ch^{k+\frac12} |u|_{k+1,l^e_1}.
\end{align*}
By \eqref{trace-polynomial}, 
we get 
$|\hat v|_{2,l^{\hat K}_1} =h^{\frac32}|\hat v|_{2,l^{e}_1} \leq  C h |v|_{2,e}. $
Thus we have
\begin{align*}
 & \sum_{e\cap L_1\neq \emptyset}\int_{-1}^1 (l_{k}\hat b_{k+1}\hat a)_I\hat v ds- \int_{-1}^1 (l_{k}\hat b_{k+1}\hat a)_I\hat v d^hs
\leq \sum_{e\cap L_1\neq \emptyset}C \|\hat a\|_{k,\infty,\hat K} |\hat b_{k+1}(1)| |\hat v|_{2,l^{\hat K}_1}\\
=& \mathcal O(h^{k+\frac32})\sum_{e\cap L_1\neq \emptyset}\|a\|_{k,\infty }|u|_{k+1, l^e_1} 
|v|_{2, e}
=\mathcal O(h^{k+\frac32})\|a\|_{k,\infty } |u|_{k+1, L_1} 
\|v\|_{2, \Omega}
=\mathcal O(h^{k+\frac32}) \|a\|_{k,\infty } \|u\|_{k+2, \Omega} 
\|v\|_{2, \Omega},
\end{align*}
where the trace inequality $ \|u\|_{k+1,\partial \Omega} \leq C  \|u\|_{k+2, \Omega}$ is used.

Combine all the estimates above, we get \eqref{crossterm-1}. Since the $\frac12$ order loss is only due to the line integral along $L_1$ and $L_3$, on which  $v_{xx}=0$ if $v\in V^h_0$, we get \eqref{crossterm-2}. 
\end{proof}

By all the discussions in this subsection, we have proven \eqref{ellip-u-up-1} and \eqref{ellip-u-up-2}.

\section{Homogeneous Dirichlet Boundary Conditions}
\label{sec-main}

\subsection{$V^h$-ellipticity}
\label{sec-vhellipticity}
In order to discuss the scheme \eqref{intro-scheme2},
we need to show $A_h$   satisfies $V^h$-ellipticity \begin{equation}
\forall v_h\in V^h_0,\quad  C\|v_h\|^2_{1}\leq A_h(v_h,v_h).
\end{equation}
We first consider the $V_h$-ellipticity  for  the case $\mathbf b\equiv 0$. 
\begin{lemma}
\label{vh-ellipticity-lemma}
Assume the coefficients in \eqref{bilinearform} satisfy that
$\mathbf b\equiv 0$, both $c(x,y)$ and the eigenvalues of $\mathbf a(x,y)$ have a uniform upper bound and a uniform positive lower bound,
 then there exist two constants $C_1, C_2>0$ independent of mesh size $h$ such that 
\begin{equation*}
\forall v_h\in V_0^h,\quad  C_1\|v_h\|^2_{1}\leq A_h(v_h,v_h)\leq C_2 \|v_h\|^2_{1}.
\end{equation*}
\end{lemma}
\begin{proof}
Let $Z_{0,\hat K}$ denote the set of $(k+1)\times (k+1)$ Gauss-Lobatto points on the reference cell $\hat K$.
First we notice that the set $Z_{0,\hat K}$ is a $Q^k(\hat K)$-unisolvent subset. Since the Gauss-Lobatto quadrature weights are strictly positive, we have 
$$\forall \hat p\in Q^k(\hat K),\, \sum_{i=1}^{2} \langle \partial_i \hat p,\partial_i \hat p\rangle_{\hat K}=0 \Longrightarrow \partial_i \hat p=0 \textrm{ at quadrature points},$$
where $i=1,2$ represents the spatial derivative on variable $x_i$ respectively.
Since $\partial_i \hat p\in Q^k(\hat K)$ and it vanishes on a $Q^k(\hat K)$-unisolvent subset, we have $\partial_i \hat p\equiv 0$. As a consequence, $\sqrt{\sum_{i=1}^{n}\langle \partial_i \hat p,\partial_i \hat p\rangle_h}$ defines a norm over the quotient space $Q^k(\hat K)/Q^0(\hat K)$. Since that $|\cdot|_{1,\hat K}$ is also a norm over the same quotient space, by the equivalence of norms over a finite dimensional space, we have
$$
\forall \hat p\in Q^k(\hat K),\quad  C_1 |\hat p|^2_{1,\hat K}\leq \sum_{i=1}^{n}\langle \partial_i \hat p,\partial_i \hat p\rangle_{\hat K} \leq C_2 |\hat p|^2_{1,\hat K}.
$$
On the reference cell $\hat K$, by the assumption on the coefficients, we have 
$$
 C_1 |\hat v_h|^2_{1,\hat K} \leq C_1  \sum_{i}^{n} \langle \partial_i \hat v_h, \partial_i \hat v_h
\rangle_{\hat K} \leq \sum_{i,j=1}^{n} \left(\langle\hat a_{ij} \partial_i \hat v_h, \partial_j \hat v_h\rangle_{\hat K}+\langle\hat c \hat v_h, \hat v_h\rangle_{\hat K}\right)\leq 
C_2 \|\hat v_h\|^2_{1,\hat K}$$

Mapping these back to the original cell $e$ and summing over all elements, by the equivalence of two norms $|\cdot|_{1}$ and $\|\cdot\|_{1}$ for the space $H^1_0(\Omega)\supset V^h_0$ \cite{ciarlet1991basic}, we get $C_1\|v_h\|^2_1 \leq A_h(v_h,v_h)\leq C_2\|v_h\|^2_1.$  
\end{proof}

 For discussing  $V_h$-ellipticity  when $\mathbf b$ is nonzero, by Young's inequality we have 
\begin{align*}
|\langle \mathbf b\cdot \nabla v_h, v_h\rangle_h| \leq \sum_e\iint_e \frac{(\mathbf b \cdot \nabla v_h)^2}{4c} + c |v_h|^2d^hxd^hy \leq 
\langle \frac{|\mathbf b|^2}{4c} \nabla v_h, \nabla v_h\rangle_h + \langle c v_h, v_h\rangle_h.
\end{align*}
Thus we have 
\begin{align*}
\langle \mathbf a \nabla v_h, \nabla v_h\rangle_h +\langle \mathbf b\cdot \nabla v_h, v_h\rangle_h + \langle c v_h, v_h\rangle_h \geq 
\langle \lambda_{\mathbf a} \nabla v_h, \nabla v_h\rangle_h - \langle \frac{|\mathbf b|^2}{4c} \nabla v_h, \nabla v_h\rangle_h,
\end{align*}
where $\lambda_{\mathbf{a}}$ is smallest eigenvalue of $\mathbf{a}$.
  Then we have the following Lemma
  \begin{lemma}
  \label{vh-ellipticity-lemma-b-nozero}
Assume  $4\lambda_{\mathbf{a}}c > |\mathbf{b}|^2$,  then there exists a constant $C>0$ independent of mesh size $h$ such that 
\begin{equation*}
\forall v_h\in V_0^h, \quad A_h(v_h,v_h) \geq C \|v_h\|^2_{1}.
\end{equation*}
  \end{lemma}

\subsection{Standard estimates for the dual problem}\label{dualpro}

In order to apply the Aubin-Nitsche duality argument for establishing superconvergence of function values, we need certain estimates on a proper 
dual problem. 
Define $\theta_h:=u_h-u_p$. 
Then we consider the dual problem:
  find $w\in H_0^1(\Omega)$ satisfying 
\begin{equation}
A^*(w,v)=(\theta_h,v),\quad \forall v\in H_0^1(\Omega),
\label{eqn-dual}
\end{equation}
where $A^*(\cdot,\cdot)$ is the adjoint bilinear form of $A(\cdot,\cdot)$ such that 
$$A^*(u,v) = A(v,u) = (\mathbf a \nabla v, \nabla u)+(\mathbf b\cdot \nabla v,u)+(c v, u).$$
 Let $w_h\in V_0^h$ be the solution to
\begin{equation}
A^*_h(w_h,v_h)=(\theta_h,v_h),\quad \forall v_h\in V_0^h.
\label{scheme-dual}
\end{equation}
Notice that the right hand side of \eqref{scheme-dual} is different from the right hand side of the scheme \eqref{intro-scheme2}.

We need the following standard estimates on $w_h$ for the dual problem. 

\begin{theorem}
\label{elliptic-reg} 
Assume all coefficients in \eqref{bilinearform} are in $W^{2,\infty}(\Omega)$,
elliptic regularity and $V^h$ ellipticity holds, we have
\[\|w-w_h\|_1\leq C h \|w\|_2,\]
$$\|w_h\|_2\leq  C\|\theta_h\|_0.$$
\end{theorem}
\begin{proof}
By  $V^h$ ellipticity,  we have $C_1\|w_h-v_h\|_1^2\leq  A^*_h(w_h-v_h, w_h-v_h)$. 
By the definition of the dual problem, we have
\[A^*_h(w_h, w_h-v_h)=(\theta_h, w_h-v_h)=A^*(w, w_h-v_h),\quad\forall v_h\in V_0^h.\]
Thus for any $v_h\in V_0^h$, by Theorem \ref{a-ah-1norm}, we have
 \begin{align*}
& C_1\|w_h-v_h\|_1^2
\leq A^*_h(w_h-v_h, w_h-v_h)\\
=&A^*(w-v_h, w_h-v_h)+[A_h^*(w_h, w_h-v_h)-A^*(w, w_h-v_h)]
+[A^*(v_h, w_h-v_h)-A^*_h(v_h, w_h-v_h)]\\
=&A^*(w-v_h, w_h-v_h)+[A(w_h-v_h, v_h)-A_h(w_h-v_h,v_h)]\\
\leq & C\|w-v_h\|_1 \|w_h-v_h\|_1+Ch \|v_h \|_2 \|w_h-v_h\|_1.
\end{align*}
Thus 
\begin{equation}\label{w-whwithvh}
\|w-w_h\|_1\leq \|w-v_h\|_1 +\|w_h-v_h\|_1 \leq C \|w-v_h\|_1+Ch\|v_h \|_2. 
\end{equation}
Now consider $\Pi_1 w \in V^h_0$ where $\Pi_1$ is the piecewise $Q^1$ projection and its definition on each cell is defined through \eqref{projection1} on the reference cell. 
By the Bramble Hilbert Lemma Theorem \ref{bh-lemma} on the projection error, we have 
\begin{equation}\label{dual-eqn-0}
\|w-\Pi_1 w\|_1\leq C h \|w\|_2,\quad \|w-\Pi_1 w\|_2\leq C \|w\|_2,
\end{equation}
thus $\|\Pi_1 w\|_2\leq \|w\|_2+\|w-\Pi_1 w\|_2\leq C \|w\|_2$.
By setting $v_h=\Pi_1 w$, from \eqref{w-whwithvh} 
we have
\begin{equation}
\|w-w_h\|_1\leq C \|w-\Pi_1 w\|_1+Ch \|\Pi_1 w \|_2\leq C h \|w\|_2. 
\label{1norm-dual}
\end{equation}
By the inverse estimate on the piecewise polynomial $w_h-\Pi_1 w$, we get 
\begin{equation}
\|w_h\|_2
\leq \|w_h-\Pi_1 w\|_2+\|\Pi_1 w-w\|_2+\|w\|_2\leq Ch^{-1}\|w_h-\Pi_1 w\|_1+C\|w\|_2.\label{dual-eqn-1}
\end{equation}
By \eqref{dual-eqn-0} and \eqref{1norm-dual}, we also have
\begin{align}
&\|w_h-\Pi_1 w\|_1\leq  \|w-\Pi_1 w\|_1+\|w-w_h\|_1\leq C h \|w\|_2.
\label{dual-eqn-2}
\end{align}

With \eqref{dual-eqn-1}, \eqref{dual-eqn-2} and the elliptic regularity
$\|w\|_2\leq C \|\theta_h\|_0$, we get  
\begin{equation*}\|w_h\|_2\leq C\|w\|_2\leq  C \|\theta_h\|_0.
\end{equation*}
\end{proof}
\subsection{Superconvergence of function values}

\begin{theorem}
\label{mainthm}
 Assume $a_{ij}, b_i, c \in W^{k+2,\infty}(\Omega)$ and $u(x,y)\in H^{k+3}(\Omega)$, $f(x,y)\in H^{k+2}(\Omega)$. Assume  $V^h$ ellipticity holds. 
 Then  $u_h$ is a $(k+2)$-th order accurate approximation to $u$ in the discrete 2-norm over all the $(k+1)\times (k+1)$ Gauss-Lobatto points:
 \[\|u_h-u\|_{2, Z_0}= \mathcal O (h^{k+2})(\|u\|_{k+3,\Omega}+\|f\|_{k+2,\Omega}).\] 
\end{theorem}
\begin{proof}
By Theorem \ref{a-ah} and Theorem \ref{rhs-estimate}, for any $v_h \in V^h_0$,
\begin{equation*}
\begin{array}{cl}
&A_h(u-u_h, v_h)=[A(u, v_h)-A_h(u_h, v_h)]+ [A_h(u, v_h) - A(u, v_h)]\\
=& A(u, v_h)-A_h(u_h, v_h)+\mathcal O(h^{k+2}) \|a\|_{k+2,\infty}\|u\|_{k+3} \|v_h\|_2 \\
=& [(f,v_h)-\langle f,v_h\rangle_h] +\mathcal O(h^{k+2}) \|u\|_{k+3} \|v_h\|_2
=\mathcal O(h^{k+2}) (\|u\|_{k+3} +\|f\|_{k+2})\|v_h\|_2.
\end{array}
\end{equation*}
Let $\theta_h=u_h-u_p$, then $\theta_h\in V_0^h$ due to the properties of the M-type projection. So by \eqref{ellip-u-up-1} and Theorem \ref{elliptic-reg}, we get
\begin{align*}
& \|\theta_h\|_0^2=(\theta_h,\theta_h)=A_h(\theta_h,w_h)= A_h(u_h-u, w_h)+A_h(u-u_p, w_h)\\
=&  A_h(u-u_p, w_h)+\mathcal O(h^{k+2})(\|u\|_{k+3} +\|f\|_{k+2})\|w_h\|_2\\
=&\mathcal O(h^{k+2})(\|u\|_{k+3} +\|f\|_{k+2})\|w_h\|_2 
=\mathcal O(h^{k+2})( \|u\|_{k+3} +\|f\|_{k+2})\|\theta_h\|_0,
\end{align*}
thus
\begin{equation*}
\|u_h-u_p\|_0=\|\theta_h\|_0=\mathcal O(h^{k+2})( \|u\|_{k+3} +\|f\|_{k+2}).
\end{equation*}
Finally, by the equivalence of the discrete 2-norm on $Z_0$ and the
$L^2(\Omega)$ norm in finite-dimensional space $V^h$ and Theorem \ref{thm-superapproximation}, we obtain 
\begin{align*}
\|u_h-u\|_{2,Z_0}\leq \|u_h-u_p\|_{2,Z_0}+\|u_p-u\|_{2,Z_0} \leq C \|u_h-u_p\|_{0}+\|u_p-u\|_{2,Z_0} \\
= \mathcal O (h^{k+2})(\|u\|_{k+3}+\|f\|_{k+2}). 
\end{align*}
\end{proof}
\begin{remark}
To extend the discussions to Neumann type boundary conditions, due to \eqref{ellip-u-up-2} and Lemma \ref{a-ah}, one can only prove $(k+\frac32)$-th order accuracy: 
 \[\|u_h-u\|_{2, Z_0}= \mathcal O (h^{k+\frac32})(\|u\|_{k+3}+\|f\|_{k+2}).\] On the other hand, for solving a general elliptic equation, only $\mathcal O(h^{k+\frac32})$ superconvergence at all Lobatto point can be proven for Neumann boundary conditions even for the full finite element scheme \eqref{intro-scheme1}, see \cite{chen2001structure}.
 \label{remark-loss}
 \end{remark}
 \begin{remark}
 All key discussions  can be
 extended to three-dimensional cases. 
 \end{remark}

\section{Nonhomogeneous Dirichlet Boundary Conditions}
\label{sec-nonhomogeneous-dirichlet}

We consider a two-dimensional elliptic problem on $\Omega=(0,1)^2$ with nonhomogeneous Dirichlet boundary condition,
\begin{equation}\label{nonhom-elliptic}
\begin{array}{cl}
 -\nabla ( \mathbf a \nabla u) +\mathbf b\cdot \nabla u+c u & =  f \textrm{ on } \Omega\\
 u &=  g \textrm{ on } \partial\Omega.
 \end{array}
\end{equation}
 Assume there is a function $\bar g \in H^1(\Omega)$ as a smooth extension of $g$ so that $\bar g|_{\partial \Omega} = g$.
 The  variational form is to find $\tilde u = u - \bar g \in H_0^1(\Omega)$ satisfying
\begin{equation}\label{nonhom-var}
 A(\tilde u, v)=(f,v) - A(\bar g,v) ,\quad \forall v\in H_0^1(\Omega).
 \end{equation}
 
 In practice, $\bar g$ is not used explicitly. By abusing notations,
the most convenient implementation is to consider
\[g(x,y)=\begin{cases}
   0,& \mbox{if}\quad (x,y)\in (0,1)\times(0,1),\\
   g(x,y),& \mbox{if}\quad (x,y)\in \partial\Omega,\\
  \end{cases}
\] 
and
$g_I\in V^h$ which is defined as the $Q^k$ Lagrange interpolation at $(k+1)\times (k+1)$ Gauss-Lobatto points for each cell on $  \Omega$ of $g(x,y)$.
Namely, $g_I\in V^h$ is the piecewise $P^k$ interpolation of $g$ along the boundary grid points and $g_I=0$ at the interior grid points. 
The numerical scheme is to  find $\tilde u_h \in V_0^h$, s.t.
\begin{equation}\label{nonhom-var-num3}
A_h( \tilde u_h, v_h)=\langle f,v_h \rangle_h - A_h( g_I,v_h) ,\quad \forall v_h\in V_0^h.
\end{equation}
Then $u_h = \tilde u_h + g_I$ will be our numerical solution for  \eqref{nonhom-elliptic}.
Notice that \eqref{nonhom-var-num3} is not a straightforward approximation to \eqref{nonhom-var} since $\bar g$ is never used. 
Assuming elliptic regularity and $V^h$ ellipticity hold, we will show that $u_h-u$ is of $(k+2)$-th order in the discrete 2-norm over all  $(k+1)\times (k+1)$ Gauss-Lobatto points. 
 
 \subsection{An auxiliary scheme} In order to discuss the superconvergence of \eqref{nonhom-var-num3}, we  need to prove the superconvergence of an auxiliary scheme. Notice that we discuss  the auxiliary scheme only for proving the accuracy of  \eqref{nonhom-var-num3}. In practice one should not implement the auxiliary scheme since  \eqref{nonhom-var-num3} is a much more convenient implementation with the same accuracy. 
 
Let $\bar g_p\in V^h$ be the piecewise M-type $Q^k$ projection of the smooth extension function $\bar g$,  and define $g_p\in V^h$ as $g_p=\bar g_p$ on $\partial \Omega$ and $g_p = 0$ at all the inner grids. 
The  auxiliary scheme is to find $\tilde u^{*}_h \in V_0^h$ satisfying
\begin{equation}\label{nonhom-var-num1}
A_h(\tilde u^{*}_h, v_h)=\langle f,v_h \rangle_h - A_h(g_p,v_h) ,\quad \forall v_h\in V_0^h,
\end{equation}

Then $u^{*}_h = \tilde u^{*}_h + g_p$ is the numerical solution of scheme \eqref{nonhom-var-num1} for problem \eqref{nonhom-var}.
Define $\theta_h=u^{*}_h-u_p$, then 
by Theorem \ref{plp-projection-theorem} we have $\theta_h \in V_0^h$.
Following Section \ref{dualpro}, define the following dual problem:
  find $w\in H_0^1(\Omega)$ satisfying 
\begin{equation}
A^*(w,v)=(\theta_h,v),\quad \forall v\in H_0^1(\Omega).
\label{eqn-dual-2}
\end{equation}
 Let $w_h\in V_0^h$ be the solution to
\begin{equation}
A^*_h(w_h,v_h)=(\theta_h,v_h),\quad \forall v_h\in V_0^h.
\label{scheme-dual-2}
\end{equation}
Notice that the dual problem has homogeneous Dirichlet boundary conditions.
By Theorem \ref{a-ah}, Theorem \ref{rhs-estimate}, for any $v_h \in V^h_0$,
\begin{equation*}
\begin{array}{cl}
&A_h(u-u^{*}_h, v_h)= [A(u, v_h)-A_h(u^{*}_h, v_h)]+ [A_h(u, v_h) - A(u, v_h)]\\
=& A(u, v_h)-A_h(u^{*}_h, v_h)+\mathcal O(h^{k+2}) \|a\|_{k+2,\infty}\|u\|_{k+3} \|v_h\|_2 \\
=& [(f,v_h)-\langle f,v_h\rangle_h] +\mathcal O(h^{k+2}) \|u\|_{k+3} \|v_h\|_2
=\mathcal O(h^{k+2}) (\|u\|_{k+3} +\|f\|_{k+2})\|v_h\|_2.
\end{array}
\end{equation*}
By \eqref{ellip-u-up-1} and Theorem \ref{elliptic-reg}, we get
\begin{align*}
& \|\theta_h\|_0^2=(\theta_h,\theta_h)=A_h(\theta_h,w_h)
= A_h(u^{*}_h-u, w_h)+A_h(u-u_p, w_h)\\
=&  A_h(u-u_p, w_h)+\mathcal O(h^{k+2})(\|u\|_{k+3} +\|f\|_{k+2})\|w_h\|_2\\
=&\mathcal O(h^{k+2})(\|u\|_{k+3} +\|f\|_{k+2})\|w_h\|_2 =\mathcal O(h^{k+2})( \|u\|_{k+3} +\|f\|_{k+2})\|\theta_h\|_0,
\end{align*}
thus $
\|u_h^{*}-u_p\|_0=\|\theta_h\|_0=\mathcal O(h^{k+2})( \|u\|_{k+3} +\|f\|_{k+2}).$
So Theorem \ref{mainthm} still holds for the auxiliary scheme \eqref{nonhom-var-num1}:
\begin{equation}
\label{estimate-aux-scheme-1}
 \|u^*_h-u\|_{2,Z_0}=\mathcal O(h^{k+2})(\|u\|_{k+3}+\|f\|_{k+2}). 
\end{equation}

\subsection{The main result}
In order to extend Theorem \ref{mainthm} to \eqref{nonhom-var-num3}, we only need to prove 
\[\|u_h-u^*_h\|_0=\mathcal O(h^{k+2}).\]
The difference between \eqref{nonhom-var-num1} and \eqref{nonhom-var-num3} is 
\begin{equation}\label{Ah-u1-u2}
A_h(\tilde u^{*}_h- \tilde u_h, v_h)= A_h( g_I-g_p,v_h) ,\quad \forall v_h\in V_0^h.
\end{equation}
We need the following Lemma.

\begin{lemma}
\label{lemma-boundary}
Assuming $u\in H^{k+4}(\Omega)$, then we have
\begin{equation} 
A_h( g_I-g_p,v_h) = \mathcal O(h^{k+2})\|u\|_{k+4,\Omega}\|v_h\|_{2,\Omega},\quad \forall v_h\in V_0^h.
\end{equation}
\end{lemma}
\begin{proof}
For simplicity, we ignore the subscript $_h$ of $v_h$ in this proof and all the following $v$ are in $V^h$.

Notice that $g_I-g_p\equiv 0$ in interior cells. Thus we only consider cells adjacent to $\partial \Omega$. Let $L_1, L_2, L_3$ and $L_4$ denote the top, left, bottom and right boundary edges of $\bar \Omega=[0,1]\times [0,1]$ respectively. Without loss of generality, we consider cell $e=[x_e-h,x_e+h]\times[y_e-h,y_e+h]$ adjacent to the left boundary $L_2$, i.e., $x_e-h=0$. Let $l^e_1, l^e_2, l^e_3$ and $l^e_4$ denote the top, left, bottom and right boundary edges of $e$ respectively.

On $l_2\subset L_2$,  Let $\phi_{ij}(x,y),i,j=0,1,\dots,k,$ be Lagrange basis functions on edge $l_2^e$ for the $(k+1)\times (k+1)$ Gauss-Lobatto points in cell $e$. Then $g_I-g_p = \sum_{i,j=0}^{k}\lambda_{ij} \phi_{ij}(x,y)$ and $|\lambda_{ij}|\leq \|g_I-g_p\|_{\infty,Z_0}$.  Due to Sobolev's embedding, we have $u\in W^{k+2,\infty}(\Omega)$. By Theorem \ref{thm-superapproximation}, we have 
\begin{align*}
\|g_I-g_p\|_{\infty,Z_0} \leq \|u-u_p\|_{\infty,Z_0}= \mathcal{O}(h^{k+2})\|u\|_{k+2,\infty,\Omega} = \mathcal{O}(h^{k+2})\|u\|_{k+4,\Omega}.
\end{align*}

Thus we get $\forall v \in V^h_0$,
\begin{align*}
\langle a (g_I-g_p)_x, v_x\rangle_e= \langle a \sum_{i,j=0}^k\lambda_{ij} \phi_{ij}(x,y)_x, v_x\rangle_e \leq C\|{a}\|_{\infty,\Omega}\max_{i,j}|\lambda_{ij}||\langle  \sum_{i,j=0}^k\phi_{ij}(x,y)_x, v_x\rangle_e |.
\end{align*}
Since for polynomials on $\hat K$ all the norm are equivalent, we have 
\begin{align*}
|\langle  \sum_{i,j=0}^k\phi_{ij}(x,y)_x, v_x\rangle_e | = |\langle  \sum_{i,j=0}^k\hat \phi_{ij}(s,t)_s, \hat v_s\rangle_{\hat K} | \leq C |\hat v_s|_{\infty,\hat K} \leq C| v|_{1,\hat K} = C |v|_{1,e},
\end{align*}
which implies
\begin{align*}
\langle a (g_I-g_p)_x, v_x\rangle_h\leq   C\|{a}\|_{\infty,\Omega}\sum_e \max_{i,j}|\lambda_{ij}||v|_{1,e}= \mathcal O(h^{k+2})\|{a}\|_{\infty,\Omega} \|u\|_{k+4,\Omega} \|v\|_{2,\Omega}
\end{align*}
Similarly, for any $v \in V^h_0$, we have
\begin{eqnarray*}
\langle a (g_I-g_p)_y, v_y\rangle_h = &&\mathcal O(h^{k+2})\|{a}\|_{\infty} \|u\|_{k+4} \|v\|_{2},\\
\langle a (g_I-g_p)_x, v_y\rangle_h = &&\mathcal O(h^{k+2})\|{a}\|_{\infty} \|u\|_{k+4} \|v\|_{2},\\
\langle \textbf{b} \cdot \nabla(g_I-g_p), v\rangle_h = &&\mathcal O(h^{k+2})\|\mathbf {b}\|_{\infty} \|u\|_{k+4} \|v\|_{2},\\
\langle c(g_I-g_p), v\rangle_h = &&\mathcal O(h^{k+2})\|{c}\|_{\infty} \|u\|_{k+4} \|v\|_{2}.
\end{eqnarray*}

Thus we conclude that 
\begin{equation*} 
A_h( g_I-g_p,v_h) = \mathcal O(h^{k+2})\|u\|_{k+4}\|v_h\|_{2},\quad \forall v_h\in V_0^h.
\end{equation*}
\end{proof}

By \eqref{Ah-u1-u2} and Lemma \ref{lemma-boundary}, we have 
\begin{equation}
\label{d-boundary-esimate}
A_h(\tilde u^{*}_h- \tilde u_h, v_h) = \mathcal O(h^{k+2})\|u\|_{k+4}\|v_h\|_{2},\quad \forall v_h\in V_0^h.
\end{equation}
Let $\theta_h=\tilde u^{*}_h- \tilde u_h \in V_0^h$. 
Following Section \ref{dualpro}, define the following dual problem:
  find $w\in H_0^1(\Omega)$ satisfying 
\begin{equation}
A^*(w,v)=(\theta_h,v),\quad \forall v\in H_0^1(\Omega).
\label{eqn-dual-3}
\end{equation}
 Let $w_h\in V_0^h$ be the solution to
\begin{equation}
A^*_h(w_h,v_h)=(\theta_h,v_h),\quad \forall v_h\in V_0^h.
\label{scheme-dual-3}
\end{equation} 
By \eqref{d-boundary-esimate} and Theorem \ref{elliptic-reg}, we get
\[\|\theta_h\|_0^2=(\theta_h,\theta_h)=A_h^*(w_h,\theta_h)
= A_h(\tilde u^{*}_h- \tilde u_h, w_h)=\mathcal O(h^{k+2})\|u\|_{k+4}\|w_h\|_{2}
=\mathcal O(h^{k+2})\|u\|_{k+4}\|\theta_h\|_{0},
\]
thus $
\|\tilde u^{*}_h- \tilde u_h\|_0=\|\theta_h\|_0=\mathcal O(h^{k+2}) \|u\|_{k+4}.$
By equivalence of norms for polynomials, we have
\begin{equation}\label{inner}
\|\tilde u^{*}_h- \tilde u_h\|_{2, Z_0}\leq C \|\tilde u^{*}_h- \tilde u_h\|_{0} = \mathcal O(h^{k+2})\|u\|_{k+4,\Omega}.
\end{equation}

Notice that both $\tilde u_h$ and $\tilde u^*_h$ are constant zero along $\partial\Omega$, and $u_h|_{\partial \Omega}=g_I$ is the Lagrangian interpolation of $g$ along $\partial\Omega$.  With \eqref{estimate-aux-scheme-1},  we have proven the following main result. 
\begin{theorem}
For a nonhomogeneous Dirichlet boundary problem \eqref{nonhom-elliptic},  with suitable smoothness assumptions $a_{ij}, b_i, c \in W^{k+2,\infty}(\Omega)$, $u(x,y)\in H^{k+4}(\Omega)$ and $f(x,y)\in H^{k+2}(\Omega)$, the numerical solution $u_h$ by scheme \eqref{nonhom-var-num3} is a $(k+2)$-th order accurate approximation to $u$ in the discrete 2-norm over all the $(k+1)\times (k+1)$ Gauss-Lobatto points:
 \[\|u_h-u\|_{2, Z_0}= \mathcal O (h^{k+2})(\|u\|_{k+4}+\|f\|_{k+2}).\] 

\end{theorem}

 \section{Finite difference implementation}
\label{sec-fd}
In this section we present the finite difference implementation of the scheme \eqref{nonhom-var-num3} for the case $k=2$ on a uniform mesh. The finite difference implementation of the nonhomogeneous Dirichlet boundary value problem is based on a homogeneous Neumann boundary value problem, which will be discussed first.
We demonstrate how it is derived for the one-dimensional case then give the two-dimensional implementation. 
 It provides efficient assembling of the stiffness matrix and one can easily implement it in MATLAB. 
 Implementations for higher order elements or quasi-uniform meshes can be similarly derived, even though it will no longer be a conventional finite difference scheme on a uniform grid. 
\subsection{One-dimensional case}
Consider a homogeneous Neumann boundary value problem
$
-(au')' =  f \textrm{ on } [0,1], u'(0) = 0,   u'(1) =  0,
$
and its variational form is to seek $u\in H^1([0,1])$ satisfying
\begin{align}
\label{1d-homo-neumann}
(au',v')=(f,v), \quad \forall v\in H^1([0,1]).
\end{align}
Consider a uniform mesh $x_i = ih$, $i = 0,1,\dots, n+1 $, $h=\frac{1}{n+1}$. Assume $n$ is odd and let $N=\frac{n+1}{2}$. Define intervals $I_k =[x_{2k},x_{2k+2}]$ for $k=0,\dots,N-1$ as a finite element mesh for $P^2$ basis. Define 
$$V^h=\{v\in C^0([0,1]): v|_{I_k}\in P^2(I_k), k = 0,\dots, N-1\}.$$ Let $\{v_i\}_{i=0}^{n+1} 	\subset V^h $ be a basis of $V^h$ such that $v_i(x_j)= \delta_{ij}, \,i,j=0,1,\dots,n+1$.
With $3$-point Gauss-Lobatto quadrature, the $C^0$-$P^2$ finite element method for \eqref{1d-homo-neumann} is to seek $u_h\in V^h$ satisfying  
\begin{align}
\label{1D-scheme-neumann}
\langle au_h',v_i'\rangle_h=\langle f,v_i\rangle_h, \quad i=0,1,\dots,n+1.
\end{align}

Let $u_j=u_h(x_j)$, $a_j=a(x_j)$ and $f_j=f(x_j)$ then $u_h(x)=\sum\limits_{j=0}^{n+1} u_jv_j(x)$. We have 
$$\sum_{j=0}^{n+1} u_j \langle a v_j',v_i'\rangle_h =\langle au_h',v_j'\rangle_h = \langle f,v_i\rangle_h=\sum_{j=0}^{n+1} f_j \langle  v_j,v_i\rangle_h , \quad i=0,1,\dots,n+1.$$
The matrix form of this scheme is $\bar S\bar{ \mathbf u}=\bar M \bar{\mathbf f}$,
where \begin{align*}
\bar{\textbf{u}}=\begin{bmatrix}
u_0,u_1,\dots,u_{n},u_{n+1}
\end{bmatrix}^T,\quad \bar{\textbf{f}}=\begin{bmatrix}
f_0,f_1,\dots,f_{n},f_{n+1}
\end{bmatrix}^T,
\end{align*}
the stiffness matrix $\bar S$ is has size $(n+2)\times(n+2)$ with $(i,j)$-th entry as $\langle a v_i',v_j'\rangle_h$, and
the lumped mass matrix $M$ is a $(n+2)\times(n+2)$ diagonal matrix
with diagonal entries $h \begin{pmatrix}
\frac13,\frac43,\frac23,\frac43,\frac23,\dots,\frac23,\frac43,\frac13
\end{pmatrix}$.

 Next we derive an explicit representation of the   matrix $\bar S$.
 Since basis functions $v_i\in V^h$ and $u_h(x)$ are not $C^1$ at the knots $x_{2k}$ ($k=1,2,\dots,N-1$), their derivatives at the knots are double valued. 
We will use superscripts $+$ and $-$ to denote derivatives obtained from the right and from the left respectively,
e.g., $v'^+_{2k}$ and $v'^-_{2k+2}$ denote the derivatives of $v_{2k}$ and $v_{2k+2}$ respectively in the interval $I_{k}=[x_{2k}, x_{2k+2}]$. Then in the interval $I_{k}=[x_{2k}, x_{2k+2}]$ we have the following representation of derivatives
 \begin{equation}
\begin{bmatrix} 
v'^+_{2k}(x)\\
v'_{2k+1}(x)\\
v'^-_{2k+2}(x)
\end{bmatrix}
= \frac{1}{2h}\begin{bmatrix}
-3 & 4 & -1\\
-1 & 0 & 1\\
1 & -4 & 3
\end{bmatrix}  
\begin{bmatrix}
 v_{2k}(x)\\
v_{2k+1}(x)\\
v_{2k+2}(x)
\end{bmatrix}.
\label{local-differentiation}  
 \end{equation}

By abusing notations, we use $(v_i)'_{2k}$ to denote the average of two derivatives of $v_i$ at the knots $x_{2k}$:
$$(v_i)'_{2k} = \frac12[(v_i')_{2k}^-+(v_i')^+_{2k}]. $$
Let $[v_i]$ denote the difference between the right derivative and left derivative:
$$[v_i']_0=[v_i']_{n+2}=0, \quad [v_i']_{2k}: = (v_i')^+_{2k}-(v_i')^-_{2k}, \quad k=1,2,\dots,N-1.$$
Then at the knots, we have
\begin{equation}
(v_i')^-_{2k} (v_j')^-_{2k}+(v_i')^+_{2k} (v_j')^+_{2k} =2(v_i')_{2k} (v_j')_{2k}+ \frac12[v_i]_{2k} [v_j]_{2k}.
\label{fd-eqn-1}
\end{equation}
 We also have
 \begin{equation}
 \langle a v_j',v_i'\rangle_{I_{2k}} =h\left[\frac13a_{2k}(v_j')^+_{2k}(v_i')^+_{2k}+\frac43a_{2k+1}(v_j')_{2k+1}(v_i')_{2k+1}+\frac13 a_{2k+2} (v_j')^-_{2k+2}(v_i')^-_{2k+2}\right].
\label{fd-eqn-2}
\end{equation}
 Let $\mathbf v_i$ denote a column vector of size $n+2$ consisting of grid point values of $v_i(x)$. Plugging \eqref{fd-eqn-1} into \eqref{fd-eqn-2}, with \eqref{local-differentiation}, we get 
 $$\langle a v_j',v_i'\rangle_h =\sum_{k=0}^{N-1} \langle a v_j',v_i'\rangle_{I_{2k}}= \frac{1}{h} \textbf{v}_i^T ( D^T WA D + E^T WA E)\textbf{v}_j,$$
 where $A$ is a diagonal matrix with diagonal entries $a_0,a_1,\dots,a_{n},a_{n+1}$, 
and
  \begin{align*}
  W=&diag\begin{pmatrix}
\frac13,\frac43,\frac23,\frac43,\frac23,\dots,\frac23,\frac43,\frac13
\end{pmatrix}_{(n+2)\times (n+2)},\\
D=&\frac{1}{2}\left(\begin{smallmatrix}
-3 & 4 & -1 & & & & & & \\
-1 & 0 & 1 & & & & & & \\
\frac12 & -2 & 0 & 2 & -\frac12 & & &  & &\\
& & -1 & 0 & 1 & & &   \\
& & \frac12 & -2 & 0 & 2 & -\frac12 & & \\
& & & & -1 & 0 & 1 &\\
& & & & & \ddots & \ddots & \ddots &  & \\
& &  & & & -1 & 0 & 1 &\\
& & & & &  \frac12 & -2 & 0 & 2 & -\frac12\\
& & & & & & & -1 & 0 & 1\\
& & & & & & & 1 & -4 & 3
\end{smallmatrix}\right)_{(n+2)\times (n+2)}, 
E=\frac{1}{2} \left(\begin{smallmatrix}
0 & 0 & 0 & & & & & & \\
0 & 0 & 0 & & & & & & \\
-\frac12 & 2 & -3 & 2 & -\frac12 & & &  & &\\
& & 0 & 0 & 0 & & &   \\
& & -\frac12 & 2 & -3 & 2 & -\frac12 & & \\
& & & & 0 & 0 & 0 &\\
& & & & & \ddots & \ddots & \ddots &  & \\
& &  & & & 0 & 0 & 0 &\\
& & & & &  -\frac12 & 2 & -3 & 2 & -\frac12\\
& & & & & & & 0 & 0 & 0\\
& & & & & & & 0 & 0 & 0
\end{smallmatrix}\right)_{(n+2)\times (n+2)}.
 \end{align*}
 Since $\{v_i\}_{i=0}^{n}$ are the Lagrangian basis for $V^h$, we have 
 \begin{equation}
\bar S =  \frac{1}{h}(D^T W A D + E^T W A E).
\label{1d-neumann-s-matrix}
 \end{equation}

Now consider the one-dimensional Dirichlet boundary value problem:
\begin{align*}
-(au')' = & f \textrm{ on } [0,1], \\
 u(0) = \sigma_1,  \quad & u(1) =  \sigma_2.
\end{align*}
Consider the same mesh as above and define 
$$V^h_0=\{v\in C^0([0,1]): v|_{I_k}\in P^2(I_k), k = 0,\dots, N-1; v(0)=v(1)=0\}.$$ Then $\{v_i\}_{i=1}^{n} 	\subset V^h $ is a basis of $V^h_0$ for $\{v_i\}_{i=0}^{n+1}$ defined above.
The one-dimensional version of \eqref{nonhom-var-num3} is to seek $u_h\in V^h_0$ satisfying  
\begin{equation}
 \begin{split}
\langle au_h',v_i'\rangle_h&=\langle f,v_i\rangle_h-\langle a g_I',v_i'\rangle_h, \quad i=1,2,\dots,n,\\
g_I(x)&=\sigma_0 v_0(x)+\sigma_1 v_{n+1}(x).
\end{split}
\label{1D-scheme-Dirichlet}
\end{equation}
Notice that we can obtain \eqref{1D-scheme-Dirichlet} by simply setting $u_h(0)=\sigma_0$ and $u_h(1)=\sigma_1$ in \eqref{1D-scheme-neumann}. So the finite difference implementation of \eqref{1D-scheme-Dirichlet} is given as follows:
\begin{enumerate}
 \item Assemble the  $(n+2)\times(n+2)$ stiffness matrix $\bar S$ for homogeneous Neumann problem as in \eqref{1d-neumann-s-matrix}. 
 \item Let $ S$ denote the $n\times n$ submatrix $\bar S(2:n+1, 2:n+1)$, i.e., $[\bar S_{ij}]$ for $i,j=2,\cdots, n+1$. 
 \item Let $\mathbf l$ denote the $n\times 1$ submatrix $\bar S(2:n+1, 1)$ and $\mathbf r$ denote the $n\times 1$ submatrix $\bar S(2:n+1, n+2)$,
 which correspond to $v_0(x)$ and $v_{n+1}(x)$. 
 \item Let $\mathbf u=\begin{bmatrix}
                       u_1 & u_2 &\cdots & u_n
                      \end{bmatrix}^T$
            and $\mathbf f=\begin{bmatrix}
                       f_1 & f_2 &\cdots & f_n
                      \end{bmatrix}^T$.
   Define $\mathbf w=\begin{bmatrix}
                      \frac43,\frac23,\frac43,\frac23,\dots,\frac23,\frac43
                     \end{bmatrix}
$ as a column vector of size $n$.                    
The scheme \eqref{1D-scheme-Dirichlet} can be implemented as 
 \[S \mathbf u= h \mathbf w^T\mathbf f -\sigma_0\mathbf l-\sigma_1\mathbf r.\]
\end{enumerate}

\subsection{Notations and tools for the two-dimensional case}

We will need two operators:
\begin{itemize}
 \item Kronecker product of two matrices: if $A$ is $m \times n$ and $B$ is $p\times q$, then
 $A\otimes B$ is $mp\times nq$ give by
 \[A\otimes B=\begin{pmatrix}
                  a_{11}B & \cdots & a_{1n} B\\
                  \vdots & \vdots & \vdots\\
                  a_{m1}B & \cdots & a_{mn} B
                 \end{pmatrix}.
\]
\item For a $m\times n$ matrix $X$, $vec(X)$ denotes  the vectorization of the matrix $X$ by rearranging $X$ into a vector column by column.\end{itemize}
The following properties will be used:
\begin{enumerate}
 \item $(A \otimes B)(C \otimes D)=AC \otimes BD$.
 \item $(A \otimes B)^{-1}=A^{-1}\otimes B^{-1}$.
 \item $(B^T\otimes A)vec(X)=vec(AXB)$.
 \item $(A \otimes B)^T=A^T \otimes B^T.$
\end{enumerate}

Consider a uniform grid $(x_i,y_j)$ for a rectangular domain $\bar \Omega=[0,1]\times[0,1]$
where
$x_i = ih_x$, $i = 0,1,\dots, n_x+1$, $h_x=\frac{1}{n_x+1}$
and
$y_j = jh_y$, $j = 0,1,\dots, n_y+1$, $h_y=\frac{1}{n_y+1}$. 

Assume $n_x$ and $n_y$ are odd and let $N_x=\frac{n_x+1}{2}$ and 
$N_y=\frac{n_y+1}{2}$. We consider rectangular cells $e_{kl} =[x_{2k},x_{2k+2}]\times
[y_{2l},y_{2l+2}]$ for $k=0,\dots,N_x-1$ and $l=0,\dots,N_y-1$ as a finite element mesh for $Q^2$ basis. Define 
$$V^h=\{v\in C^0(\Omega): v|_{e_{kl}}\in Q^2(e_{kl}), k = 0,\dots, N_x-1, l = 0,\dots, N_y-1 \},$$ 
$$V^h_0=\{v\in C^0(\Omega): v|_{e_{kl}}\in Q^2(e_{kl}), k = 0,\dots, N_x-1, l = 0,\dots, N_y-1; v|_{\partial \Omega}\equiv 0 \}.$$

For the coefficients $\mathbf a(x,y)=\begin{pmatrix}
               a^{11} & a^{12}\\
               a^{21} & a^{22}
              \end{pmatrix}
$, $\mathbf b=[b^1 \quad b^2]$ and $c$ in the elliptic operator \eqref{bilinearform}, 
consider their grid point values in the following form:
\begin{align*}
   A^{kl}=\begin{pmatrix}
   a_{00} & a_{01} & \dots & a_{0,n_x+1}\\
   a_{10} & a_{11} & \dots & a_{1,n_x+1}\\
   \vdots & \vdots & & \vdots \\
      a_{n_y+1,0} & a_{n_y+1,1} & \dots & a_{n_y+1,,n_x+1}
   \end{pmatrix}_{(n_y+2)\times (n_x+2)},\quad a_{ij}=a^{kl}(x_j,y_i), \quad k,l =1,2, 
   \end{align*}
\begin{align*}
   B^{m}=\begin{pmatrix}
   b_{00} & b_{01} & \dots & b_{0,n_x+1}\\
   b_{10} & b_{11} & \dots & b_{1,n_x+1}\\
   \vdots & \vdots & & \vdots \\
   b_{n_y+1,0} & b_{n_y+1,1} & \dots & b_{n_y+1,n_x+1}
   \end{pmatrix}_{(n_y+2)\times (n_x+2)},\quad b_{ij}=b^{m}(x_j,y_i), \quad m =1,2,
\end{align*}
\begin{align*}
   C=\begin{pmatrix}
   c_{00} & c_{01} & \dots & c_{0,n_x+1}\\
   c_{10} & c_{11} & \dots & c_{1,n_x+1}\\
   \vdots & \vdots & & \vdots \\
   c_{n_y+1,0} & c_{n_y+1,1} & \dots & c_{n_y+1,n_x+1}
   \end{pmatrix}_{(n_y+2)\times (n_x+2)},\quad c_{ij}=c(x_j,y_i).
\end{align*}

Let $diag(\mathbf x)$ denote a diagonal matrix with the vector $\mathbf x$ as diagonal entries
and define \[
        \bar{W}_x=diag\begin{pmatrix}
\frac13,\frac43,\frac23,\frac43,\frac23,\dots,\frac23,\frac43,\frac13
\end{pmatrix}_{(n_x+2)\times (n_x+2)},\]
\[   \bar{W}_y=diag\begin{pmatrix}
\frac13,\frac43,\frac23,\frac43,\frac23,\dots,\frac23,\frac43,\frac13
\end{pmatrix}_{(n_y+2)\times (n_y+2)},
\]
\[
        {W}_x=diag\begin{pmatrix}
\frac43,\frac23,\frac43,\frac23,\dots,\frac23,\frac43
\end{pmatrix}_{n_x \times n_x}, {W}_y=diag\begin{pmatrix}
\frac43,\frac23,\frac43,\frac23,\dots,\frac23,\frac43
\end{pmatrix}_{n_y \times n_y }.
\]
Let $s=x$ or $y$, we define the $D$ and $E$ matrices with dimension ${(n_s+2)\times (n_s+2)}$ for each variable: 
  \begin{align*}
D_s=\frac{1}{2}\left( \begin{smallmatrix}
-3 & 4 & -1 & & & & & & \\
-1 & 0 & 1 & & & & & & \\
\frac12 & -2 & 0 & 2 & -\frac12 & & &  & &\\
& & -1 & 0 & 1 & & &   \\
& & \frac12 & -2 & 0 & 2 & -\frac12 & & \\
& & & & -1 & 0 & 1 &\\
& & & & & \ddots & \ddots & \ddots &  & \\
& &  & & & -1 & 0 & 1 &\\
& & & & &  \frac12 & -2 & 0 & 2 & -\frac12\\
& & & & & & & -1 & 0 & 1\\
& & & & & & & 1 & -4 & 3
\end{smallmatrix}\right), \quad 
E_s=\frac{1}{2}\left( \begin{smallmatrix}
0 & 0 & 0 & & & & & & \\
0 & 0 & 0 & & & & & & \\
-\frac12 & 2 & -3 & 2 & -\frac12 & & &  & &\\
& & 0 & 0 & 0 & & &   \\
& & -\frac12 & 2 & -3 & 2 & -\frac12 & & \\
& & & & 0 & 0 & 0 &\\
& & & & & \ddots & \ddots & \ddots &  & \\
& &  & & & 0 & 0 & 0 &\\
& & & & &  -\frac12 & 2 & -3 & 2 & -\frac12\\
& & & & & & & 0 & 0 & 0\\
& & & & & & & 0 & 0 & 0
\end{smallmatrix}\right).
 \end{align*}
 
 Define an inflation operator
$Infl: \mathbbm R^{n_y\times n_x}\longrightarrow \mathbbm R^{(n_y+2)\times (n_x+2)}$ by adding zeros:
\[Infl(U)=\begin{pmatrix}
        0 & \cdots & 0\\
        \vdots & U & \vdots\\
       0 & \cdots & 0\\
        \end{pmatrix}_{(n_y+2)\times (n_x+2)}
\]
and its matrix representation is given as $\tilde I_x \otimes \tilde I_y$ where
\[\tilde I_x=\begin{pmatrix}
                                             \mathbf 0 \\ I_{n_x\times n_x} \\ \mathbf 0
                                            \end{pmatrix}_{(n_x+2)\times n_x},
                                            \tilde I_y=\begin{pmatrix}
                                             \mathbf 0 \\ I_{n_y\times n_y} \\ \mathbf 0
                                            \end{pmatrix}_{(n_y+2)\times n_y}.
\]
Its adjoint is
a restriction operator $Res: \mathbbm R^{(n_y+2)\times (n_x+2)} \longrightarrow\mathbbm R^{n_y\times n_x} $ as 
$$Res(X)=X(2:n_y+1, 2:n_x+1)\quad, \forall X \in \mathbbm R^{(n_y+2)\times (n_x+2)},$$  and its matrix representation is 
$\tilde I_x^T \otimes \tilde I_y^T.$

\subsection{Two-dimensional case}
For $\bar \Omega=[0,1]^2$ we first consider 
an elliptic equation with homogeneous Neumann boundary condition: 
\begin{align}
 -\nabla \cdot(\mathbf a \nabla u  ) +\mathbf b\nabla u + c u = & f \textrm{ on } \Omega, \label{ellipticeqn}\\
 \mathbf a \nabla u \cdot \mathbf{n} = & 0 \textrm{ on } \partial\Omega.\label{fd-implementation-hom-neu}
 \end{align}
 The  variational form is to find $u\in H^1(\Omega)$ satisfying
\begin{equation}\label{varpro1}
 A( u, v)=(f,v),\quad \forall v\in H^1(\Omega).
\end{equation}
The $C^0$-$Q^2$ finite element  method with $3\times 3$ Gauss-Lobatto quadrature is to find $u_h\in V^h$ satisfying 
\begin{equation}
\langle \mathbf a \nabla u_h ,\nabla v_h \rangle_h +\langle\mathbf b\nabla u_h ,v_h\rangle_h+ \langle c u_h,v_h\rangle_h=\langle f,v_h\rangle_h,\quad \forall v_h\in V^h,
\label{2d-fd-neumann}
\end{equation}

 Let $\bar U$ be a $(n_y+2)\times(n_x+2)$ matrix such that its $(j,i)$-th entry is $\bar U(j,i)=u_h(x_{i-1},y_{j-1})$, $i=1,\dots,n_x+2$, $j=1,\dots,n_y+2$. Let $\bar F$ be a $(n_y+2)\times(n_x+2)$ matrix such that its $(j,i)$-th entry is $\bar F(j,i)=f(x_{i-1},y_{j-1})$. Then the matrix form of \eqref{2d-fd-neumann} is\begin{equation}
\label{2d-fd-neumann-2}
\bar Svec(\bar U) = \bar M vec(\bar F), \quad \bar M=h_xh_y \bar W_x\otimes \bar W_y, \quad \bar S= \sum_{k,l=1}^2 S_a^{kl}+ \sum_{m=1}^2 S_b^m  +S_c,
\end{equation}
where 
\begin{align*}
 &S_a^{11}=\frac{h_y}{h_x}(D_x^T\otimes I_y)diag(vec(\bar W_y A^{11}\bar W_x))(D_x\otimes I_y)+ \frac{h_y}{h_x}(E_x^T \otimes I_y)diag(vec(\bar W_y A^{11}\bar W_x))( E_x\otimes I_y), \\
 &S_a^{12}=(D_x^T\otimes I_y)diag(vec(\bar W_y A^{12}\bar W_x))(I_x\otimes D_y)+ (E_x^T \otimes I_y)diag(vec(\bar W_y A^{12}\bar W_x))(I_x\otimes E_y), \\
&S_a^{21}=(I_x\otimes D_y^T)diag(vec(\bar W_y A^{21}\bar W_x))(D_x\otimes I_y)+ (I_x\otimes E_y^T)diag(vec(\bar W_y A^{21}\bar W_x))( E_x\otimes I_y),\\\
&S_a^{22}=\frac{h_x}{h_y}(I_x\otimes D_y^T)diag(vec(\bar W_y A^{22}\bar W_x))(I_x\otimes D_y)+ \frac{h_x}{h_y}(I_x\otimes E_y^T)diag(vec(\bar W_y A^{22}\bar W_x))(I_x\otimes E_y),\\
&S_b^1  = h_y diag(vec(\bar W_y B^1\bar W_x))(D_x\otimes I_y),\quad 
S_b^2  = h_x diag(vec(\bar W_y B^2\bar W_x))(I_x\otimes D_y),\\
&S_c  = h_x h_y diag(vec(\bar W_y C\bar W_x).
\end{align*}

Now consider the scheme \eqref{nonhom-var-num3} for nonhomogeneous Dirichlet boundary conditions. 
Its numerical solution can be represented as a matrix $U$ of size $ny\times nx$ with $(j,i)$-entry $U(j,i)=u_h(x_i, y_j)$ for $i=1,\cdots, nx; j=1,\cdots, ny$. Similar to the one-dimensional case, its stiffness matrix can be obtained as the submatrix of $\bar S$ in \eqref{2d-fd-neumann-2}.
 Let $\bar G$ be a $(n_y+2)$ by $(n_x+2)$ matrix with $(j,i)$-th entry as $\bar G(j,i)=g(x_{i-1},y_{j-1})$, where 
 \[g(x,y)=\begin{cases}
   0,& \mbox{if}\quad (x,y)\in (0,1)\times(0,1),\\
   g(x,y),& \mbox{if}\quad (x,y)\in \partial\Omega.\\
  \end{cases}
\] 
 In particular, $\bar G(j+1,i+1) = 0$ for $j=1,\dots,n_y$, $i=1,\dots,n_x$.
 Let  $F$ be a matrix of size $ny\times nx$ with $(j,i)$-entry as $F(j,i)=f(x_i, y_j)$
 for $i=1,\cdots, nx; j=1,\cdots, ny$.
 Then  the scheme \eqref{nonhom-var-num3}
becomes
\begin{equation}
\label{2d-fd-scheme-nonhomo-dirichlet}
(\tilde I_x^T\otimes \tilde I_y^T)\bar S(\tilde I_x\otimes \tilde I_y)vec(U) = (W_x\otimes W_y)vec(F)-(\tilde I_x^T\otimes \tilde I_y^T)\bar S vec(\bar G).
\end{equation}
Even though the stiffness matrix is given as $S=(\tilde I_x^T\otimes \tilde I_y^T)\bar S(\tilde I_x\otimes \tilde I_y)$, $S$ should be implemented as a linear operator  in iterative linear system solvers. For example, the matrix vector multiplication $(\tilde I_x^T\otimes \tilde I_y^T)S^{11}_a(\tilde I_x\otimes \tilde I_y)vec(U)$ is equivalent to the following linear operator from $\mathbbm R^{ny\times nx}$ to $\mathbbm R^{ny\times nx}$:
\[ \frac{h_y}{h_x}\tilde I_y^T\left\{I_y\left([\bar W_y A^{11}\bar W_x]\circ[I_y(\tilde I_y U \tilde I^T_x)D_x^T]\right)D_x+I_y\left([\bar W_y A^{11}\bar W_x]\circ[I_y(\tilde I_y U \tilde I^T_x)E_x^T]\right)E_x\right\}\tilde I_x,\]
where $\circ$ is the Hadamard product (i.e., entrywise multiplication).  
\subsection{The Laplacian case}
\label{sec-Laplacian}
For one-dimensional constant coefficient case with homogeneous Dirichlet boundary condition, the scheme can be written as a classical finite difference scheme  $H \mathbf u=\mathbf f$ with 
\[H=M^{-1}S=\frac{1}{h^2}\left(\begin{smallmatrix}
   2& -1 & & & & &\\
   -2& \frac72 &-2 & \frac14 & & &\\
     &  -1 & 2& -1 & & &\\
   & \frac14 &-2& \frac72 &-2 & \frac14 & \\
  & &   &  -1 & 2& -1 & \\
  & & & &\ddots &\ddots & \\
  && & \frac14 &-2& \frac72 &-2\\
   &  & & & &-1 & 2\\
  \end{smallmatrix}\right)\]
In other words,  if $x_i$ is a cell center, the scheme is
  \[\frac{-u_{i-1}+2u_i-u_{i+1}}{h^2}=f_i,\]
and if $x_i$ is a knot away from the boundary, the scheme is
\begin{equation*}
\frac{u_{i-2}-8u_{i-1}+14u_i-8u_{i+1}+u_{i+2}}{4h^2}=f_i.
\end{equation*}
It is straightforward to verify that the local truncation error is only second order. 

For the two-dimensional Laplacian case homogeneous Dirichlet boundary condition, the 
scheme can be rewritten as 
\[(H_x\otimes I_y)+(I_x\otimes H_y)vec(U)=vec(F),\]
where $H_x$ and $H_y$ are the same $H$ matrix above with size $n_x\times n_x$ and $n_y\times n_y$ respectively.
The inverse of $(H_x\otimes I_y)+(I_x\otimes H_y)$ can be efficiently constructed via the eigen-decomposition of small matrices $H_x$ and $H_y$:
\begin{enumerate}
 \item Compute eigen-decomposition of $H_x=T_x \Lambda_x T_x^{-1}$ and $H_y=T_y \Lambda_y T_y^{-1}$.
 \item The properties of Kronecker product imply that 
 \[(H_x\otimes I_y)+(I_x\otimes H_y)=(T_x\otimes T_y)(\Lambda_x\otimes I_y+I_x\otimes \Lambda_y)(T_x^{-1}\otimes T_y^{-1}),\]
 thus 
  \[[(H_x\otimes I_y)+(I_x\otimes H_y)]^{-1}=(T_x\otimes T_y)(\Lambda_x\otimes I_y+I_x\otimes \Lambda_y)^{-1}(T_x^{-1}\otimes T_y^{-1}).\]
  \item It is nontrivial to determine whether $H$ is diagonalizable. In all our numerical tests, $H$
has no repeated eigenvalues. So if assuming $\Lambda_x$ and $\Lambda_y$ are diagonal matrices, the matrix vector multiplication $[(H_x\otimes I_y)+(I_x\otimes H_y)]^{-1}vec(F)$ can be implemented as a linear operator on $F$:
\begin{equation}
 T_y([T_y^{-1} F(T_x^{-1})^T]./\Lambda)T_x^T, 
 \label{2d-preconditioner}
 \end{equation}
where $\Lambda$ is a $n_y\times n_x$ matrix with $(i,j)$-th entry as $\Lambda(i,j)=\Lambda_y(i,i)+\Lambda_x(j,j)$
and $./$ denotes entry-wise division for two matrices of the same size. 
\end{enumerate}

For the 3D Laplacian, the  matrix can be represented as $H_x\otimes I_y\otimes I_z+I_x\otimes H_y\otimes I_z+I_x\otimes I_y\otimes H_z$ thus can be efficiently inverted through eigen-decomposition of small matrices $H_x, H_y$ and $H_z$ as well.

Since the eigen-decomposition of small matrices $H_x$ and $H_y$ can be precomputed, and \eqref{2d-preconditioner} costs only $\mathcal O(n^3)$ for a 2D problem on a mesh size $n\times n$, in practice \eqref{2d-preconditioner} can be used as a simple preconditioner in conjugate gradient solvers for 
the following linear system equivalent to \eqref{2d-fd-scheme-nonhomo-dirichlet}:
\[ (W_x^{-1}\otimes W_y^{-1})(\tilde I_x^T\otimes \tilde I_y^T)\bar S(\tilde I_x\otimes \tilde I_y)vec(U) = vec(F)-(W_x^{-1}\otimes W_y^{-1})(\tilde I_x^T\otimes \tilde I_y^T)\bar S vec(G), \]
even though the multigrid method as reviewed in \cite{xu2017algebraic} is the optimal solver in terms of computational complexity.

\section{Numerical results}
\label{sec-test}

In this section we show a few numerical tests verifying the accuracy of the scheme \eqref{nonhom-var-num3} for $k=2$ implemented as a finite difference scheme on a uniform grid. 
We first consider the following two dimensional elliptic equation:  
\begin{equation}
 - \nabla\cdot(\mathbf a\nabla u)+\mathbf b\cdot\nabla u+c u=f\quad \textrm{on } [0,1]\times [0,2]
 \label{test-eqn-1}
\end{equation}
where $\mathbf a=\left( {\begin{array}{cc}
   a_{11} & a_{12} \\
   a_{21} & a_{22} \\
  \end{array} } \right)$, $a_{11}=10+30y^5+x\cos{y}+y$, $a_{12}=a_{21}=2+0.5(\sin(\pi x)+x^3)(\sin(\pi y)+y^3)+\cos(x^4+y^3)$, $a_{22}=10+x^5$, $\mathbf b=\mathbf 0$, $c=1+x^4y^3$, with an exact solution
  $$u(x,y)=0.1(\sin(\pi x)+x^3)(\sin(\pi y)+y^3)+\cos(x^4+y^3).$$

The errors at grid points are listed in Table  \ref{dirichlet_elliptic} for purely Dirichlet boundary condition and Table \ref{neumann_elliptic} for purely Neumann boundary condition. We observe fourth order accuracy in the discrete 2-norm for both tests, even though only $\mathcal O(h^{3.5})$ can be proven for Neumann boundary condition  as discussed in Remark \ref{remark-loss}.  Regarding the maximum norm of the superconvergence of the function values at Gauss-Lobatto points,  one can only prove $\mathcal O(h^3\log h)$ even for the full finite element scheme \eqref{intro-scheme1}
since discrete Green's function is used, see \cite{chen2001structure}.

 \begin{table}[h]
\label{dirichlet_elliptic}
\centering
\caption{A 2D elliptic equation with Dirichlet boundary conditions. The first column is the number of regular cells in a finite element mesh. The second column is the number of grid points in a finite difference implementation, i.e., number of degree of freedoms. }
\begin{tabular}{|c |c |c c|c c|}
\hline  FEM Mesh &  FD Grid & $l^2$ error  &  order & $l^\infty$ error & order \\
\hline
 $2\times 4$ & $3\times 7$ & 3.94E-2 & - & 7.15E-2 & -\\
\hline
 $4\times 8$&  $7\times 15$& 1.23E-2 & 1.67 & 3.28E-2 & 1.12\\
\hline
 $8\times 16$& $15\times 31$ & 1.46E-3 & 3.08 & 5.42E-3 & 2.60 \\
\hline
$16\times 32$ & $31\times 63$ & 1.14E-4 & 3.68 & 3.96E-4 & 3.78 \\ 
 \hline
 $32\times 64$ &  $63\times 127$ & 7.75E-6 & 3.88 & 2.62E-5 & 3.92\\
  \hline
 $64\times 128$ &  $127\times 255$ & 5.02E-7 & 3.95 & 1.73E-6 & 3.92\\
   \hline
 $128\times 256$&   $255\times 511$ & 3.23E-8 & 3.96 & 1.13E-7 & 3.94\\
\hline
\end{tabular}
\end{table}

\begin{table}[h]
\label{neumann_elliptic}
\centering
\caption{A 2D elliptic equation with Neumann boundary conditions.}
\begin{tabular}{|c |c|c c|c c|}
\hline FEM Mesh &FD Grid  &  $l^2$ error  &  order & $l^\infty$ error & order \\
\hline
 $2\times 4$& $5\times 9$ & 1.38E0 & - & 2.27E0 & -\\
\hline
$4\times 8$  & $9\times 17$ & 1.46E-1 & 3.24 & 2.52E-1 & 3.17\\
\hline
$8\times 16$  & $17\times 33$ & 7.49E-3 & 4.28 & 1.64E-2 & 3.94 \\
\hline
$16\times 32$ & $33\times 65$& 4.31E-4 & 4.12 & 1.02E-3 & 4.01 \\ 
 \hline
 $32\times 64$ &$65\times 129$ & 2.61E-5 & 4.04 & 7.47E-5 & 3.78\\
\hline
\end{tabular}
\end{table}

Next we consider a three-dimensional problem $-\Delta u=f$ with homogeneous Dirichlet boundary conditions on a cube $[0,1]^3$ with the following exact solution 
$$u(x,y,z)=\sin(\pi x)\sin(2\pi y)\sin(3\pi z)+(x-x^3)(y^2-y^4)(z-z^2).$$
See Table \ref{dirichlet_elliptic3D} for the performance of the finite difference scheme. There is no essential difficulty to extend the proof to three dimensions, even though  it is not very straightforward. Nonetheless we observe that the scheme is indeed fourth order accurate. The linear system is solved by the eigenvector method shown in Section \ref{sec-Laplacian}. The discrete 2-norm over the set of all grid points $Z_0$ is defined as 
$\|u\|_{2,Z_0}=\left[h^3\sum_{(x,y,z)\in Z_0} |u(x,y,z)|^2\right]^{\frac12}$.

\begin{table}[h]
\label{dirichlet_elliptic3D}
\centering
\caption{$-\Delta u=f$ in 3D with homogeneous Dirichlet boundary condition.}
\begin{tabular}{|c|c c|c c|}
\hline Finite Difference Grid   &  $l^2$ error  &  order & $l^\infty$ error & order \\
\hline
 $7\times 7\times 7$  & 1.51E-2 & - & 4.87E-2 & -\\
\hline
 $15\times 15\times 15$  & 9.23E-4 & 4.04 & 3.12E-3 & 3.96 \\
\hline
$31\times 31\times 31$   & 5.68E-5 & 4.02 & 1.95E-4 & 4.00 \\ 
 \hline
 $63\times 63\times 63$   & 3.54E-6 & 4.01 & 1.22E-5 & 4.00\\
  \hline
 $127\times 127\times 127$  & 2.21E-7 & 4.00 & 7.59E-7 & 4.00\\
\hline
\end{tabular}
\end{table}

Last we consider \eqref{test-eqn-1} with convection term and the coefficients $\mathbf b$ is incompressible $\nabla\cdot \mathbf b=0$: $\mathbf a=\left( {\begin{array}{cc}
   a_{11} & a_{12} \\
   a_{21} & a_{22} \\
  \end{array} } \right)$, $a_{11}=100+30y^5+x\cos{y}+y$, $a_{12}=a_{21}=2+0.5(\sin(\pi x)+x^3)(\sin(\pi y)+y^3)+\cos(x^4+y^3)$, $a_{22}=100+x^5$, $\mathbf b=\left( {\begin{array}{c}
   b_{1} \\
   b_{2} \\
  \end{array} } \right)$, $b_1=\psi_y$, $b_2=-\psi_x$, $\psi = x\exp(x^2+y)$, $c=1+x^4y^3$, with an exact solution
  $$u(x,y)=0.1(\sin(\pi x)+x^3)(\sin(\pi y)+y^3)+\cos(x^4+y^3).$$

The errors at grid points are listed in Table \ref{dirichlet_elliptic_withconvec} for   Dirichlet boundary conditions.

  \begin{table}[h]
\centering
\caption{A 2D elliptic equation with convection term and Dirichlet boundary conditions.}
\begin{tabular}{|c |c |c c|c c|}
\hline  FEM Mesh &  FD Grid & $l^2$ error  &  order & $l^\infty$ error & order \\
\hline
 $2\times 4$ & $3\times 7$ & 1.26E-1 & - & 2.71E-1 & -\\
\hline
 $4\times 8$&  $7\times 15$& 2.85E-2 & 2.15 & 9.70E-2 & 1.48\\
\hline
 $8\times 16$& $15\times 31$ & 1.89E-3 & 3.92 & 7.25E-3 & 3.74 \\
\hline
$16\times 32$ & $31\times 63$ & 1.17E-4 & 4.01 & 4.01E-4 & 4.17 \\ 
 \hline
 $32\times 64$ &  $63\times 127$ & 7.41E-6 & 3.98 & 2.54E-5 & 3.98\\
  \hline
\end{tabular}
\label{dirichlet_elliptic_withconvec}
\end{table}
\section{Concluding remarks}
\label{sec-conclusion}
In this paper we have proven the superconvergence of function values in the simplest finite difference implementation of $C^0$-$Q^k$ finite element method for elliptic equations. In particular, for the case $k=2$ the scheme \eqref{nonhom-var-num3} can be easily implemented as a fourth order accurate finite difference scheme as shown in Section \ref{sec-fd}. It provides only only an convenient approach for constructing fourth order accurate finite difference schemes but also the most efficient implementation of $C^0$-$Q^k$ finite element method without losing superconvergence of function values. 
In a follow up paper \cite{li2019dmp}, we will show that discrete maximum principle can be proven for
the scheme \eqref{nonhom-var-num3} in the case $k=2$ when solving a variable coefficient Poisson equation.

\bibliographystyle{siamplain}
\bibliography{ref.bib}

\end{document}